\documentclass[11pt]{amsart}

\usepackage[utf8]{inputenc}
\usepackage[margin=2.9cm]{geometry}
\usepackage{xcolor}
\usepackage{graphicx}
\usepackage{todonotes}
\usepackage{subcaption}
\usepackage{mathtools}

\mathtoolsset{showonlyrefs}

\usepackage{amsmath, amssymb, amsthm}
\usepackage{esint}
\usepackage{verbatim}
\usepackage{dsfont}
\usepackage{enumitem}
\usepackage{helvet}

\DeclareMathOperator{\dist}{dist}

\DeclareMathOperator{\dive}{div}

\DeclareMathOperator{\conv}{conv}
\DeclareMathOperator{\supp}{supp}

\newcommand{\Lm}{\mathcal{L}} 
\newcommand{\N}{\mathbb{N}} 
\newcommand{\Z}{\mathbb{Z}} 
\newcommand{\R}{\mathbb{R}} 
\newcommand{\T}{\mathbb{T}} 
\newcommand{\E}{\mathbb{E}} 
\newcommand{\Prob}{\mathbb{P}} 
\newcommand{\M}{{\R^d_x}}
\newcommand{\TM}{{\R^d_x\times \R^d_v}}

\newcommand{\id}{\mathrm{id}}
\newcommand{\PM}{\mathcal{P}}

\newcommand{\weakstar}{\stackrel{*}{\rightharpoonup}}
\newcommand{\eps}{\varepsilon}

\newcommand{\rel}{\mathrm{rel}}
\newcommand{\inc}{\mathrm{inc}}

\newcommand{\dualbra}[2]{\left\langle #1,#2 \right\rangle}
\newcommand{\norm}[2][]{\left\|#2\right\|_{#1}}

\newtheorem{prop}{Proposition}[section]

\newtheorem{lemma}[prop]{Lemma}
\newtheorem{theorem}[prop]{Theorem}
\newtheorem{rem}[prop]{Remark}
\newtheorem{defi}[prop]{Definition}
\newtheorem{cor}[prop]{Corollary}
\newtheorem{example}[prop]{Example}

\title{Variational interacting particle systems and Vlasov equations}
\date{\today}

\makeatletter
\@namedef{subjclassname@2020}{\textup{2020} Mathematics Subject Classification}
\makeatother

\begin{document}

\author[P. Gladbach]{Peter Gladbach}
\address{Peter Gladbach\\ Institut f\"ur Angewandte Mathematik \\ Rheinische Friedrich-Wilhelms-Universit\"at Bonn \\ Endenicher Allee 60 \\ 53115 Bonn, Germany}
\email{gladbach@iam.uni-bonn.de}

\author[B. Kepka]{Bernhard Kepka}
\address{Bernhard Kepka\\ Institut f\"ur Angewandte Mathematik \\ Rheinische Friedrich-Wilhelms-Universit\"at Bonn \\ Endenicher Allee 60 \\ 53115 Bonn, Germany}
\email{kepka@iam.uni-bonn.de}


\begin{abstract}
We consider optimization problems for interacting particle systems. We show that critical points solve a Vlasov equation, and that in general no minimizers exist despite continuity of the action functional. We prove an explicit representation of the relaxation of the action functional. We show convergence of N-particle minimizers to minimizers of the relaxed action, and finally characterize minimizers of dynamic interacting particle optimal transport problems as solutions to Hamilton-Jacobi-Bellman equations.
\end{abstract}

\maketitle

\tableofcontents

\section{Introduction}
The study of variational interacting particle systems goes back to the work of William Rowan Hamilton \cite{hamilton1834general} describing different classical interacting particles using the minimal action formulation: For example, ions with charges $q_i\in \R$ and masses $m_i>0$ moving through space will collectively minimize (at least locally) the Lagrangian action functional
\begin{equation}\label{eq: particle action}
\int_0^T \left\lbrace -\sum_{i=1}^N\sum_{j\neq i} \frac{q_iq_j}{|\gamma_i(t) - \gamma_j(t)|} + \sum_{i=1}^N \frac{m_i}{2} |\dot \gamma_i(t)|^2 \right\rbrace\, dt ,
\end{equation}
subject to fixed initial and final positions for every particle $\gamma_i(0) = x_i^0$, $\gamma_i(T) = x_i^T$.

Minimal actions similar to \eqref{eq: particle action} also appear in e.g. behavioral science \cite{beaver2021overview,bouffanais2016design}, where the particles represent eusocial animals or an organized group of people or machines, as well as in numerical optimization, so called particle swarm optimization \cite{kennedy1995particle}, where a finite number of parameter candidates move simultaneously through an energy landscape.

In all these settings, the situation becomes significantly simpler if the action depends only on the \emph{mean field}, the statistics of position-velocity pairs. In this case, the action \eqref{eq: particle action} can be rewritten as
\begin{equation}
\int_0^T \Phi(P_{t,\dot t})\,dt,
\end{equation}
where $\Phi:\PM(\TM)\to \R$, and $P_{t,\dot t} = \frac1N \sum_{i=1}^N \delta_{(\gamma_i(t),\dot \gamma_i(t))}\in \PM(\TM)$ is the particle statistics at almost every time $t\in[0,T]$.

The mean field assumption is appropriate whenever we can expect particles to be indistinguishable, e.g. when they are all helium atoms and to a lesser extent certain animals.

This assumption allows us to pass to the limit for large numbers of particles, the so-called mean field limit, by focusing only on the statistics of particles instead of individual paths. Famously, this method was used by Lanford to derive the Boltzmann equation \cite{lanford2005time}.

In this article, we treat particles as completely nonatomic, meaning we admit any probability measure on the path space $PS_T^p = W^{1,p}([0,T];\M)$, which allows for continuous splitting and merging of particles.
We shall describe the general situation where $\Phi:\PM_p(\TM) \to \R$ is a Wasserstein-continuous energy with linear growth in the $p$-moment of $P_{t,\dot t}$, with $p\in(1,\infty)$. This covers a wide range of interactions, the typical physical one being the Vlasov-type, where
\begin{equation}\label{eq: Vlasov energy}
	\Phi(f) = \int_{\TM}\int_{\R^d_{x'}\times \R^d_{v'}} \frac12U(x-x') f(dx,dv) f(dx',dv')+ \int_{\TM} \frac12 |v|^2P_{t,\dot t}(dx,dv),
\end{equation}
for a smooth, symmetric positional interaction potential $U:\M\to \R$. Here, particles interact only through their positions, while the kinetic action is purely additive without interactions.

We show in Proposition \ref{prop:ELEquation} and in Remark \ref{rem: Vlasov} that the statistics of critical points of the action are weak solutions to a version of the Vlasov equation. As a special case, we obtain
\begin{theorem}
Let $U\in C_c^\infty(\M)$ be symmetric, i.e. $ U(x)= U(-x) $. Let $\Gamma_b \in \PM_2(\M \times \M)$ be a coupling with finite second moment. Then there exists a minimizer $P\in \PM(PS_T^2)$ of $\int_0^T \Phi(P_{t,\dot t})\,dt$, with $\Phi$ as in \eqref{eq: Vlasov energy}, subject to the boundary condition $P_{0,T} = \Gamma_b$. Its statistics $P_{t,\dot t}(dx,dv)$ weakly solve the Vlasov equation
\begin{equation}
\partial_t P_{t,\dot t} + \dive_x(vP_{t,\dot t}) + \dive_v((\nabla_x U\ast P_t) P_{t,\dot t}) =0.
\end{equation}

\end{theorem}

To our knowledge, this is the first variational characterization of solutions to the Vlasov equation, at least for smooth positive interaction potentials. 

Let us mention that on the other hand a Hamiltonian formulation of the Vlasov equation in the Wasserstein space $ (\PM_2(\TM), W_2) $ was formulated in \cite{Ambrosio2008HamiltODEs}. There, the differentiable structure of the Wasserstein space was used to prove existence of so-called Hamiltonian flows under weak regularity assumptions. The terminology ``Hamiltonian'' is justified by the fact that one can equip the Wasserstein space with a symplectic structure, see \cite{GangboPacini2011DiffFormsWasserstein}. This point of view also allowed to study the limit of the $ N $-particle system to the Vlasov–Monge–Amp\`ere system, see \cite{CullenGangboPisante2007SemigeostrophicEqDiscretized}. Furthermore, let us note that in \cite{GangboTudorascu2009EulerPoisson} the one-dimensional Euler-Poisson equation was studied using a Lagrangian formulation on the Wasserstein space.

However, this article treats not only the Lagrangian \eqref{eq: Vlasov energy}, but more general interactions between particles depending also on their velocities, e.g.

\begin{equation}\label{eq: interaction introduction}
	\Phi(f) = \int_{\TM}\int_{\R^d_{x'}\times \R^d_{v'}} \frac12U(x-x',v-v') f(dx,dv) f(dx',dv')+ \int_{\TM} \psi(x,v)P_{t,\dot t}(dx,dv),
\end{equation}

Examples include:

\begin{enumerate}
	\item Nonlocal congestion: In \eqref{eq: interaction introduction}, take
	\[
	U(x,v) = e^{-|x|^2}, \quad \psi(x,v) = \frac12 |v|^2.
	\]
	In order to minimize $\int_0^T \Phi(P_{t,\dot t})\,dt$, particles will avoid high concentrations regardless of velocity. Replacing $e^{-|x-x'|^2}$ with $|x-x'|^2$ yields the opposite effect: Particles want to be near other particles. See Figure \ref{fig:pairwiseinteracting1} for a numerical optimization.
	\item Flocking: In \eqref{eq: interaction introduction}, take 
	\[	
		U(x,v) = e^{-|x|^2}(|v|^2-1), \quad \psi(x,v) = \frac12 |v|^2.
	\] 
	In order to minimize $\int_0^T \Phi(P_{t,\dot t})\,dt$, particles moving in similar directions will flock, while particles moving in vastly different directions will avoid each other. See Figure~\ref{fig:pairwiseinteracting2} for a numerical optimization. Note that the interaction depends not only on the particles' positions but also their velocities.
\end{enumerate}

\begin{figure}[!ht]
	\centering
	\includegraphics[width=0.7\linewidth]{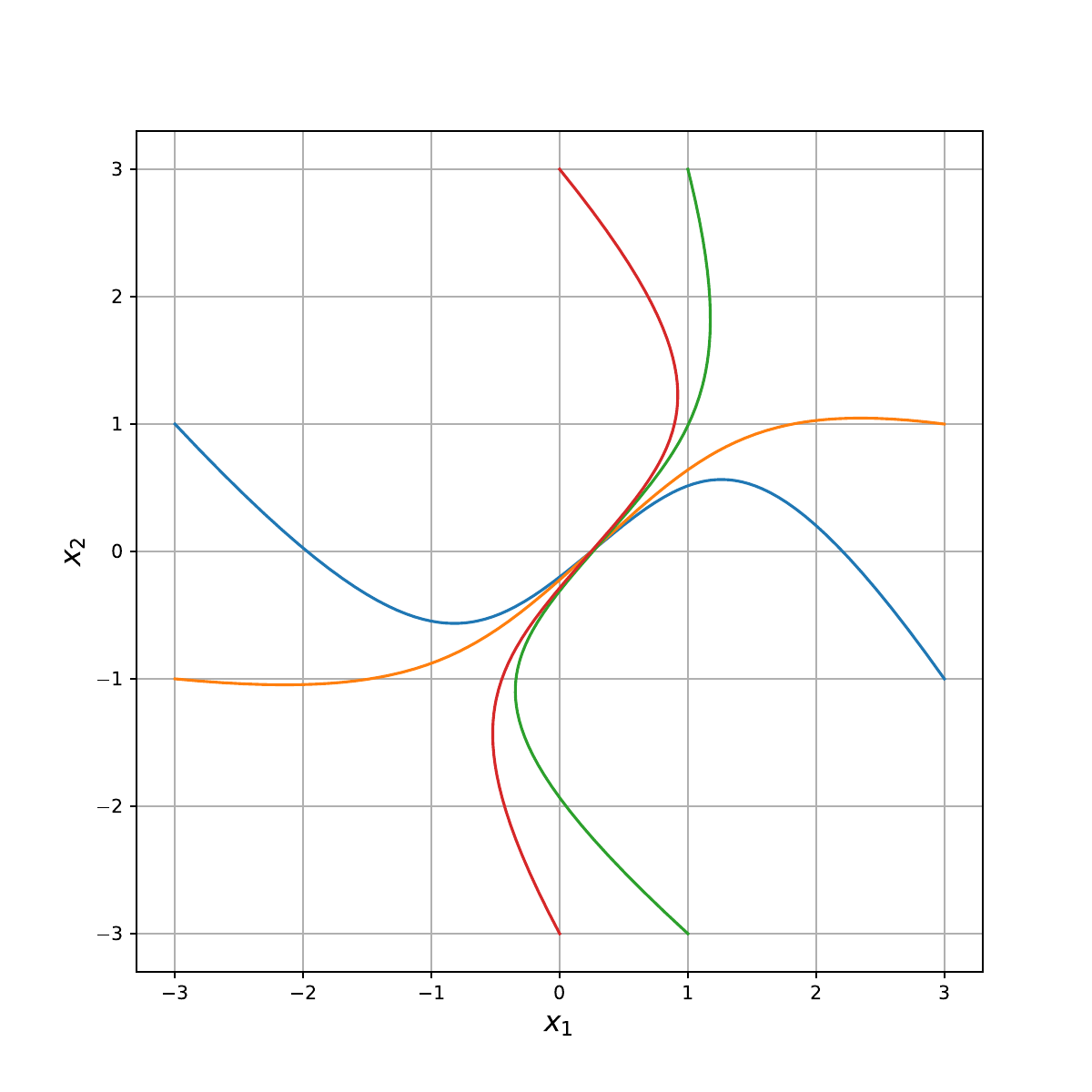}
	\caption{Optimal paths of four interacting particles in two space dimensions for $ \psi(v)= \frac{1}{2}|v|^2 $ and interacting potential $ U(x)= 50|x|^2 $. The initial and final configuration is given. Two particles move from left two right, while two particles move from bottom to top. The interaction potential leads to the effect that the particles move close to each other in order to reduce the total cost.}
	\label{fig:pairwiseinteracting1}
\end{figure}

\begin{figure}[!ht]
	\centering
	\includegraphics[width=0.7\linewidth]{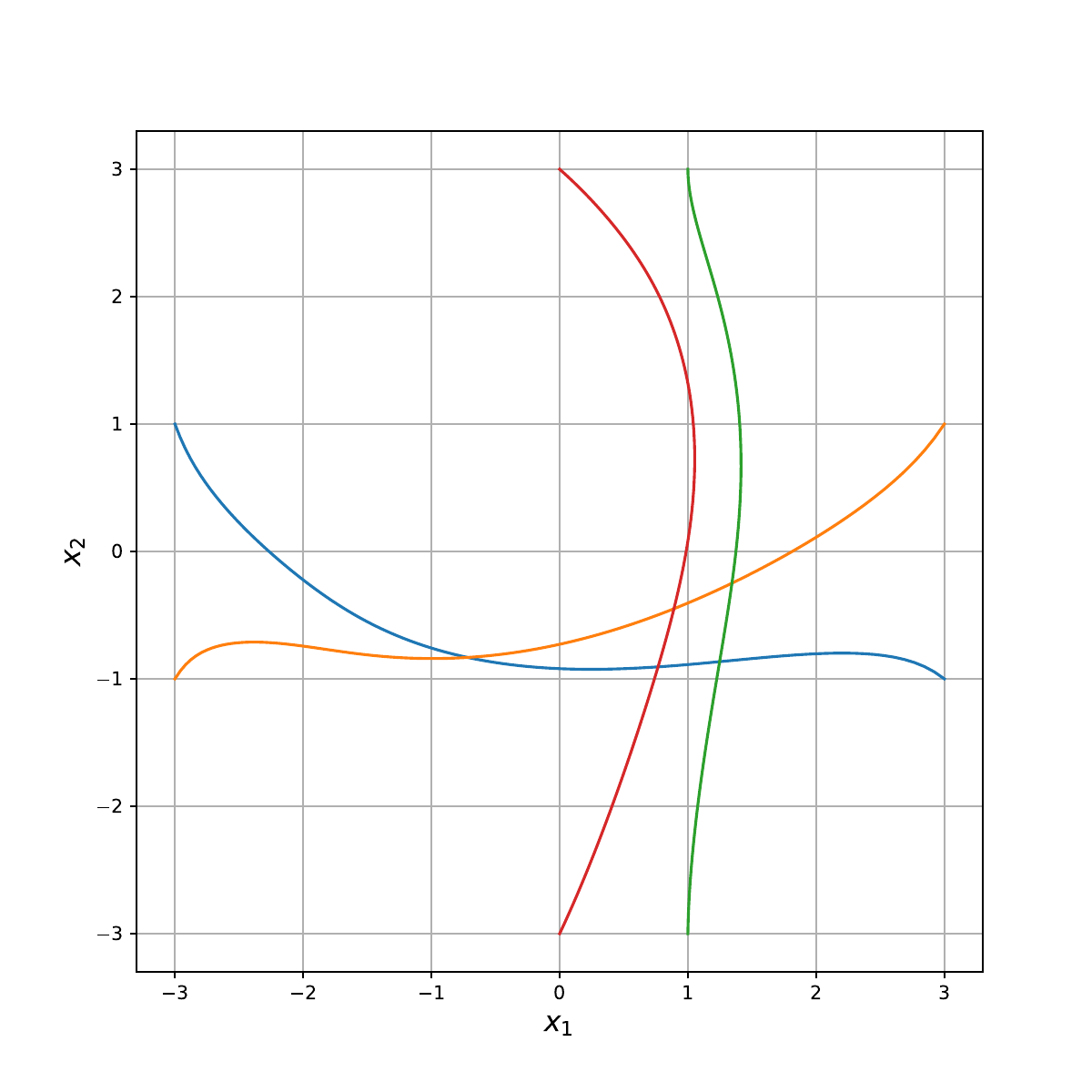}
	\caption{Optimal paths of four interacting particles in two space dimensions for $ \psi(v)= \frac{1}{2}|v|^2 $ and interacting potential $ U(x)= 50(|v|^2-10) \exp(-|x|^2) $. The initial and final configuration is prescribed. Two particles move from left two right, while two particles move from bottom to top. Due to the interaction close particles favor to align their velocities. Furthermore, the two particles moving to the right avoid the interaction with the particles moving to the top.}
	\label{fig:pairwiseinteracting2}
\end{figure}

More complex interactions are possible. We cover additional examples in Section \ref{subsec: examples}.

Existence of of minmizers of the action $\int_0^T \Phi(P_{t,\dot t})\,dt$ is not guaranteed, even for noninteracting particles, because the map $P\mapsto P_{t,\dot t}$ is not continuous. For example, the linear action
\begin{equation}
	\int_0^T \langle P_{t,\dot t}(dx,dv), |v|^4-|v|^2 + |x|^2 \rangle\,dt
\end{equation}
has no minimizer. This leads to the natural question of the relaxation of the action $\int_0^T \Phi(P_{t,\dot t})\,dt$, and sufficient and necessary conditions for lower semi-continuity. For linear actions as above, the arguments of \cite{dacorogna2014introduction} show that relaxation is achieved by computing the convex envelope of the test function in the $v$-variable. We show how to find the general relaxation formula in Theorem \ref{theorem: relaxation}, along with some interesting examples.

The relaxation formula makes use of martingale kernels in the velocity variable. Optimization over martingale kernels appears in a wide variety of applications, e.g. in mechanism design, see \cite{daskalakis2013mechanism}, and in martingale optimal transport, see \cite{beiglbock2016problem} and the references therein.

As a direct consequence of the relaxation result, we show that minimizers of the $N$-particle system converge to minimizers of the relaxed action. 
\begin{theorem}\label{theorem: N body}
Let $\Phi:\PM_p(\TM)\to \R$ satisfy assumptions (A1) and (A2) from Subsection \ref{subsec:RelaxationFormula} below. Let $\Gamma_N = \frac{1}{N}\sum_{i=1}^N\delta_{(x_0^i,x_T^i)}\in \PM(\M\times \M)$ be a sequence of couplings converging in $W_p$ to some coupling $\Gamma_b\in \PM_p(\M\times \M)$. Then any sequence of minimizers $P_N = \frac1N \sum_{i=1}^N \delta_{x^i}\in \PM(PS_T)$ to the $N$-body problem
\begin{equation}
\min \left\{ \int_0^T \Phi\left(\frac1N\delta_{(x_i(t),\dot x_i(t))}\right)\,dt\,:\,  P = \frac1N \sum_{i=1}^N \delta_{x^i}, P_{0,T} =\Gamma_N\right\}
\end{equation}
has a subsequence converging to a minimizer $P\in \PM(PS_T)$ of the relaxed continuous problem
\begin{equation}
	\min\left\{ \int_0^T \Phi^\rel(P_{t,\dot t})\,dt\,:\,P_{0,T} = \Gamma_b\right\}.
\end{equation}
\end{theorem}

Together with Proposition \ref{prop:ELEquation}, this shows that the statistics of the $N$-particle system converge to solutions of the Vlasov equation, as seen in e.g. \cite{dobrushin1979vlasov}. 

Finally, our result relates to optimal transport. While in our setting the coupling $ \Gamma_b $ is predetermined, in optimal transport the coupling is optimized, see e.g. \cite{ambrosio2013user}. In the optimal transport community, our problem is also called the who-goes-where problem. Our theory provides a natural extension of optimal transport theory for interacting particles, by first performing the relaxation for the who-goes-where-problem, and then optimizing over the coupling.

As a direct consequence, we show in Theorem \ref{theorem: optimal transport} that there is a nonlocal Eulerian characterization of the relaxation of interacting particle optimal transport problems, which uses only a single Eulerian velocity field $V:\R^d_x\to \R^d_v$. In addition, Proposition \ref{prop: HJB equation} shows that the optimal velocity field can be computed from a potential solving a Hamilton-Jacobi-Bellman equation. As a special case, if $\Phi(f) = -V(f^x)+ \langle f, \frac12 |v|^2 \rangle$, with $V:\PM(\M)\to \R$, we obtain the Hamilton-Jacobi-Bellman equation from \cite{gangbo2008hamilton}. Specifying $V(f^x) = \langle f^x,\psi \rangle + \frac{\hbar^2}{8} I(f^x)$, where $I$ is the Fisher information, yields the linear Schr\"odinger equation, see \cite{von2012optimal}.

\subsection{Organization of the article} In Section \ref{sec:SettingResultsExamples} we introduce the setting as well as notations. Furthermore, we give the main relaxation result together with the Euler-Lagrange equation. We end the section with specific examples. The Cauchy problem in the case of quadratic functionals is studied in Section \ref{sec:CauchyGenValsovEq}. Then, in Section \ref{sec:PropertiesRelaxFunctional} we state useful properties of the relaxed functional and include counterexamples for features not inherited by the relaxed functional from the unrelaxed one. In Section \ref{sec: relaxation proof} we then give the proof of the relaxation formula. The limit of the $ N $-particle problem (see Theorem \ref{theorem: N body}) is prove in Section \ref{sec:NParticleProblem}. Finally, we prove an Eulerian formulation of the interacting particle optimal transport problem in Section \ref{sec:InteractingOTP}.

\section{Setting, relaxation result and examples}\label{sec:SettingResultsExamples}
\subsection{Notation}\label{subsec:Notation}  Given a topological space $X$ we denote $\PM(X)$ the class of Borel probability measures on $X$.

We denote by $ \PM_p(\R^d) $ the set of all probability measures with finite moment of order $ p\in(1,\infty) $. We equip it with the standard $ p $-Wasserstein distance $ W_p $.

For some fixed time horizon $T>0$ we write $PS_T^p:=PS_T^p(\M):=W^{1,p}([0,T];\M)$ for the path space, where $ p\in (1,\infty) $. We equip the path space $ PS_T^p(\M) $ with the weak $W^{1,p}$-topology. We denote by $\TM$ the tangent bundle of the position space $\M$.

Given $P\in \PM(PS_T^p(\M))$ we denote by $P_t:=(\gamma \mapsto \gamma(t))_\#P\in \PM_p(\M)$ its position statistics at time $t\in [0,T]$, $P_{s,t} := (\gamma \mapsto (\gamma(s),\gamma(t)))_\#P\in \PM_p(\M\times \M)$ the joint position statistics at times $s,t\in[0,T]$. In order to obtain the statistics in space and velocity space  at time $t$ we cannot use the definition $ (\gamma \mapsto (\gamma(t),\dot \gamma(t)))_\#P\in \PM_p(\TM)$. Indeed, the paths are continuous by Sobolev embeddings; however, their derivatives are not, so that the evaluation $ t\mapsto \dot \gamma(t) $ is not well-defined. However, we can define for almost every $t\in[0,T]$ a probability measure $P_{t,\dot t}\in \PM_p(\TM)$ via 
\begin{align*}
	\int_0^T \langle P_{t,\dot t}, \varphi(t,\cdot) \rangle \,dt = \int_{PS_T^p(\M)} \int_0^T \varphi(t,\gamma(t),\dot{\gamma}(t))\, dt \, dP(\gamma)
\end{align*}
for any bounded continuous test function $ \varphi\in C^0_b([0,T]\times \TM)$, where we used the disintegration theorem \cite[Theorem 5.3.1]{ambrosio2005gradient}. We call the probability measure $P_{t,\dot t}$ the \emph{particle statistics} at time $t\in[0,T]$. We will also write $f_t := P_{t,\dot t}$ where convenient.

For a particle statistics $ f\in  \PM(\TM) $ we write $ f^x\in \PM(\M) $ for its marginal in the position variable.

In addition, we define Markov kernels in the velocity as jointly measurable maps $\pi:\TM \to \PM(\R^d_v)$. We can apply a Markov kernel to a probability measure $ f\in \PM(\TM) $ by defining $ f\pi\in \PM(\TM) $ as the measure
\begin{align*}
	f\pi(dx,dv):=  \int_{\TM}f(dx,dv')\pi(x,v',dv),
\end{align*}
which tested against a function $\varphi\in C_b^0(\TM)$ yields
\begin{align*}
	\langle f\pi, \varphi \rangle  = \int_{\TM} \int_{\R^d_{v'}} \varphi(x,v')\pi(x,v,dv')f(dx,dv).
\end{align*}

\subsection{The relaxation result}\label{subsec:RelaxationFormula} We will look for minimizers $P\in \PM(PS_T^p(\M))$ of the action functional
\begin{equation}\label{eq:Functional}
	F(P):= \int_0^T \Phi(P_{t,\dot t})\,dt,
\end{equation}
where $\Phi:\PM_p(\TM)\to \R$ is an energy functional depending only on the particle statistics at time $t$, with the following $p$-growth and $W_p$-continuity properties, with  constants $c,C>0$:
\begin{itemize}
	\item [\textbf{(A1)}] $\displaystyle{\dualbra{f}{-C+c |v|^p} \leq \Phi(f) \leq \dualbra{f}{C+C |v|^p}}$\quad for all $f\in \PM_p(\TM)$.
	\vspace*{0.2cm}
	\item [\textbf{(A2)}] $|\Phi(f) - \Phi(f')| \leq \dualbra{f+ f'}{C+C|v|^p}^{(p-1)/p} W_p(f,f')$ \quad for all $f,f'\in \PM_p(\TM)$.
\end{itemize}
In this article we are concerned with the following minimization problem
\begin{align*}
	\inf\{F(P)\,:\,P_{0,T}=\Gamma_b\},
\end{align*}
where $\Gamma_b\in \PM_p(\M\times \M)$ is a coupling with finite $p$-moment.

\begin{rem}	
Since we equip $ PS_p^T:= PS^p_T(\M):=W^{1,p}([0,T];\M) $ with the weak topology, the existence of minimizers via the direct method is not guaranteed. Indeed, take the sequence $P^n := \delta_{t\mapsto \frac1n\sin(nt)}\in \PM(PS_{2\pi}^p(\R_x))$, which has weak limit $P^n \rightharpoonup \delta_{t\mapsto 0}\in \PM(PS_{2\pi}^p(\R_x))$, for any $p\in(1,\infty)$. In contrast, $P^n_{t,\dot t} = \delta_{(\frac1n\sin(nt),\cos(nt))}\in \PM_p(\R_x \times \R_v)$ does not converge in $W_p$ to $P_{t,\dot t} = \delta_{(0,0)}$ for any $t\in[0,2\pi]$.

Taking $\Phi(f) := \psi(\langle f,v \rangle)$ for $\psi:\R\to \R$ strictly concave in $[-1,1]$, we obtain $F(P) > \limsup_{n\to\infty} F(P^n)$.
\end{rem}

As a consequence, the functional \eqref{eq:Functional} is in general not lower semi-continuous when $ \PM(PS^p_T) $ is equipped with the weak topology, see examples in Section \ref{sec:PropertiesRelaxFunctional}. The first main result of this article, Theorem \ref{theorem: relaxation}, is the identification of the relaxation of \eqref{eq:Functional}.

In order to state the relaxation formula we define martingale kernels.
\begin{defi}\label{def:MartingaleKernel}
	A Markov kernel $ \pi: \TM\to \PM_p(\R^d_v) $ is called a martingale kernel if $ \int_{\R^d_{v'}}v'\pi(x,v,dv')= v $ for all $ (x,v)\in \TM $, i.e. if the center of momentum is not shifted. We denote by $ MK_p(\TM) $ the class of martingale kernels.
\end{defi}

We can now give the relaxation formula of \eqref{eq:Functional}.
\begin{theorem}\label{theorem: relaxation}
	Let $F$ be given by \eqref{eq:Functional} with $ \Phi $ satisfying (A1), (A2). Then the relaxation of $F$ in the weak topology of $\PM(PS_T^p)$ is given by $F^\rel(P):=\int_0^T \Phi^\rel(P_{t,\dot t})\,dt$, where $\Phi^\rel:\PM_p(\TM)\to \R$ is defined by
	\begin{align}\label{eq:Relaxation}
		\Phi^\rel(f):= \inf\left\lbrace \sum_{i=1}^N \lambda_i\Phi(f\pi_i)\,:\,  \pi_i:TM\to \PM_p(\TM),\, \lambda_i\in [0,1], \right.
		\\
		\left. \sum_{i=1}^N\lambda_i=1, \, \sum_{i=1}^N\lambda_i\pi_i\in MK_p(\TM) \right\rbrace .
	\end{align}
\end{theorem}

\begin{rem}
	Consider a minimizing sequence of probability measures on the path space $P^n\weakstar P$. By the growth condition (A1), the position statistics $P^n_t\in \PM(\M)$ will converge narrowly to $P_t$, while  the same cannot be said for the position-velocity statistics $P^n_{t,\dot t}\in \PM(\TM)$. In fact, the $P^n_{t,\dot t}$ may exhibit temporal oscillations as well as additional noise in the velocity. This is the reason why martingale kernels appear in the relaxation. The kernels $\pi_i$ can be interpreted as velocity Young measures of these minimizing sequences.
	
	Alternative formulas for the relaxation are given in Lemma \ref{lem:AltFormRelaxation}.
\end{rem}



We postpone the proof of Theorem \ref{theorem: relaxation} to Section \ref{sec: relaxation proof}. Having successfully relaxed the action, minimizers of the relaxed functional exist by the direct method of the calculus of variation.

\begin{cor}\label{cor: minimizer}
    Let $\Gamma_b\in \PM_p(\M\times \M)$ be a coupling with finite $p$-moment. Then the problem $\min\{F^\rel(P)\,:\,P_{0,T}=\Gamma_b\}$ has a minimizer $P\in \PM(PS^p_T)$.
\end{cor}

\subsection{The convex order and Strassen's Theorem}
Related to velocity martingale kernels is the velocity-convex order of probability measures:
\begin{defi}\label{def:ConvexOrder}
	We write $ g\succeq f $ for $ f,\, g\in \PM(\TM) $ if $ \dualbra{g}{\varphi}\geq \dualbra{f}{\varphi} $ for all $\varphi\in C(\TM)$ that are convex in $v$ and satisfy $|\varphi(x,v)| \leq C(1+|x|+|v|)^p$. We call a functional $ \Phi: \PM_p(\TM)\to \R $ increasing (in the $v$-convex order) if $ \Phi(g)\geq \Phi(f) $ for all $ g\succeq f $.
\end{defi}
\begin{rem}
	By considering test functions $\varphi(x,v) = \varphi(x)$ that are constant in $v$ and their negatives, we obtain that $g^x = f^x$ if $g\succeq f$.
	
	Due to Strassen's Theorem \cite{Strasse1965ExiProbMeas}, the relation $ g\succeq f $ in Definition \ref{def:ConvexOrder} is equivalent to the existence of a martingale kernel $ \pi \in MK_p(\TM) $, according to Definition \ref{def:MartingaleKernel}, with $ f\pi=g $.
\end{rem}

\subsection{The Euler-Lagrange equation}
Before giving the proofs of the main results in the previous section we consider specific examples to illustrate these results. More precisely, we consider the Euler-Lagrange equations of functionals of the form \eqref{eq:Functional} and extract properties of critical points.

In order to derive the Euler-Lagrange equation of the functional \eqref{eq:Functional} let us first define appropriate variations. Let $ \varphi: PS_T^p\to C^\infty_c((0,T); \M) $ be a pathwise variation and define for $\eps\in \R$ the perturbation of paths $\Pi_\eps : PS_T^p\to PS_T^p, \gamma \mapsto \gamma +\eps \varphi_\gamma$. Note that the perturbation $ \Pi_\eps\gamma$ has the same endpoints as $ \gamma $. The corresponding perturbation of a measure $ P\in \PM(PS_T^p) $ is given by $ P_\eps:= (\Pi_\eps)_\# P $.

We say that $ P\in\PM(PS_T^p) $ is a \emph{critical point} of the functional $ F $ in \eqref{eq:Functional} if $ \eps=0 $ is a critical point of $ \eps\mapsto F(P_\eps) $ for all bounded measurable variations $ \varphi: PS_T^p\to C^1_c((0,T); \M) $.

For the purpose of identifying and solving the Euler-Lagrange equations, we restrict our attention to quadratic functionals of the form
\begin{align}\label{eq:ExampleFunctional}
	\Phi(f) := \dualbra{f}{\psi}+\dfrac{1}{2}\dualbra{f}{U* f},
\end{align} 
where $ \psi,\, U \in C^\infty(\TM;\R) $ and
\begin{align*}
	(U* f)(x,v) := \int_{\R^d_{x'}\times \R^d_{v'}} U(x-x',v-v') \, f(dx',dv').
\end{align*}
We assume that $ U $ is symmetric in the sense that $ U(x,v)= U(-x,-v) $.

The condition of critical points yields for $ f_t:=P_{t,\dot{t}} $,
\begin{align*}
	0 = \frac{d}{d\eps} F(P_\eps) =  \int_{PS_{T}^p}\int_0^T &\left[\nabla_x\psi(\gamma_t,\dot{\gamma}_t) + \nabla_x(U* f_t)(\gamma_t,\dot{\gamma}_t) \right] \cdot \varphi_\gamma (t)\\
	&+ \left[\nabla_v\psi(\gamma_t,\dot{\gamma}_t) + \nabla_v(U* f_t)(\gamma_t,\dot{\gamma}_t) \right] \cdot \dot{\varphi_\gamma} (t)\, dt \, P(d\gamma).
\end{align*}

Since the variation $ \varphi_\gamma $ was arbitrary, we conclude the following result after integration by parts:
\begin{prop}\label{prop:ELEquation}
	A measure $ P\in \PM(PS_T^p) $ is a critical point of the functional \eqref{eq:Functional}, with $ \Phi $ of the form \eqref{eq:ExampleFunctional}, if and only if $P$ is supported on paths $ \gamma \in PS_T^p$ satisfying the equation
	\begin{align}\label{eq:PathELEquation}
		\dfrac{d}{dt}\left[\nabla_v L[f_t](\gamma_t,\dot{\gamma}_t) \right] -\nabla_xL[f_t](\gamma_t,\dot{\gamma}_t) =0
	\end{align}
	in the sense of distributions $\mathcal{D}'((0,T);\M)$, where $ L[f](x,v):= \psi(x,v)+(U*f)(x,v) $.
\end{prop}

Solutions to the above are referred to as Wardrop equilibria in the literature, see e.g. \cite{carlier2008optimal} for an application in congested optimal transport or \cite{benamou2017variational} for an application in mean-field games. We note that a minimizer of the toal action is achieved by each particle minimizing the marginal cost, which is different from each particle minimizing its own contribution to the total cost.

\begin{rem}\label{rem: Vlasov}
	Let us give several remarks on the Euler-Lagrange equations.
	\begin{enumerate}
		\item Every minimizer of $F(P)$ under a given coupling $\Gamma_b\in \PM_p(\M\times \M)$ is a critical point. If $F=F^\rel$, Corollary \ref{cor: minimizer} then ensures that a critical point exists.
		
		\item If $P\in \PM(PS_T^p)$ is a critical point of $F$, the family $ f_t:=P_{t,\dot{t}} $, $ t\in[0,T] $ of statistics solves in the sense of distributions a generalized Vlasov equation 
		\begin{align}\label{eq:GenVlasovEq}
			\begin{cases}
				\partial_t f_t + \dive_x(v f_t) + \dive_v(A[f_t] f_t) = 0,
				\\
				f\mid_{t=0} = f_0
			\end{cases}
		\end{align}
		with some initial condition $ f_0 $. Let us mention that finding the acceleration field $ A[f_t]:\TM \to \R^d_a $ is not a straightforward operation. In fact, by the chain rule
		\begin{align*}
			&\frac{d}{dt}[\nabla_v L[f_t](\gamma_t,\dot \gamma_t)]\\
			= & \nabla_v\nabla_v L[f_t]\ddot{\gamma}_t + \nabla_x \nabla_v L[f_t] \dot \gamma_t + \nabla_v U * \partial_t f_t.
		\end{align*}

		We observe that $\partial_t f_t$ appears in the equation.

		Assuming $\nabla_v\nabla_v L[f]>0$, we solve for the acceleration $A[f_t](\gamma_t,\dot \gamma_t) = \ddot{\gamma_t}$, yielding the implicit equation
		\begin{align*}
			&\nabla_v\nabla_v L[f_t] A[f_t]\\
			 = & \nabla_x L[f_t] - \nabla_x\nabla_vL[f_t]v - \nabla_v U *\partial_t f_t\\
			= & \nabla_x L[f_t] - \nabla_x\nabla_vL[f_t]v + \nabla_v U *(\dive_x (v f_t)+ \dive_v(A[f_t]f_t))\\
			= & \nabla_x L[f_t] - \nabla_x\nabla_vL[f_t]v + \nabla_x \nabla_v U *(v f_t)+ \nabla_v\nabla_v U * (A[f_t]f_t),
		\end{align*}
		which after separating the terms that contain $A[f_t]$ from those that do not yields the implicit characterization of the acceleration $A[f]:\R^d_x\times \R^d_v\to\R^d_v$,
		\begin{equation}\label{eq: acceleration}
\nabla_v\nabla_v L[f] A[f] - \nabla_v\nabla_v U * (A[f]f) =  \nabla_x L[f] - \nabla_x\nabla_vL[f]v + \nabla_x \nabla_v U *(v f).
		\end{equation}
		
		In Section \ref{sec:CauchyGenValsovEq}, we show that under a convexity assumption on $\psi,\, U$ we can uniquely solve the above equation for $A[f]$ and solve the Cauchy problem for the resulting Vlasov equation \eqref{eq:GenVlasovEq} using the approach of Dobrushin \cite{dobrushin1979vlasov}.
	
		\item Finally, let us note that equations \eqref{eq:PathELEquation} are equivalent to criticality of $ P $. However, not every $P$ for which $f_t:=P_{t,\dot t}$ solves the Vlasov equation \eqref{eq:GenVlasovEq} is a critical point of $F$.
		
		For an example, take $\psi(x,v) = |v|^2-|x|^2$ for $x\in[-1,1]$ and $U=0$. Take $P:= \frac12(\delta_{t\mapsto |\sin t|} + \delta_{t\mapsto -|\sin t|})$. One easily checks that $A[f](x,v) = -x$, that $f_t$ solves \eqref{eq:GenVlasovEq} but no curve in $\supp P$ solves \eqref{eq:PathELEquation} distributionally.
	\end{enumerate}
\end{rem}

\subsection{Specific examples}\label{subsec: examples}
In this subsection, we present some examples of energy functionals $\Phi:\PM_p(\TM)\to \R$ and their respective relaxations and dynamics.

\begin{example}[Noninteracting particles]
If $\Phi(f) = \langle f, \psi \rangle$ for some $\psi\in C^1(\TM)$, we can rewrite $F(P) = E_P[\int_0^T \psi(\gamma_t,\dot \gamma_t)\,dt]$. It is clear that (A1),(A2) correspond to the following conditions on $\psi$:
\begin{align*}
	-C+c|v|^p &\leq \psi(x,v) \leq C+C|v|^p
	\\
	|\nabla \psi(x,v)| &\leq C+C|v|^{p-1}.
\end{align*}

The relaxation of $\Psi$ is $\Psi^\rel(f) = \langle f, \psi^{**} \rangle$, where $\psi^{**}$ is the $v$-convex envelope of $\psi$. Formula \eqref{eq:Relaxation} agrees with this by Strassen's theorem, see e.g. \cite{Strasse1965ExiProbMeas}.

The Euler-Lagrange equation is the classical one with acceleration
\[
A(x,v) = (\nabla_v\nabla_v\psi)^{-1}(-\nabla_x\nabla_v\psi v + \nabla_x\psi),
\]
 the Vlasov equation \eqref{eq:GenVlasovEq} is linear since $A$ does not depend on $f$. 
 \end{example}

\begin{example}[Newtonian particles]
	If 
	\[
	\Phi(f) :=  \langle f, \psi \rangle + \frac12\langle f, U*f\rangle,
	\]
	with $U\in C^1_b(\M)$ depending only on $x$ and $\psi\in C(\TM)$, we find that (A1), (A2) are satisfied in exactly the same case as for noninteracting particles. The relaxation is given by $\Phi^\rel(f) = \dualbra{f}{\psi^{**}} + \frac12\langle f,U*f\rangle$.
	
	If $\psi(x,v) = \frac12 |v|^2$, the acceleration is simply $A[f](x,v) =  \nabla_x U *f$ which is independent of $v$, yielding the standard Vlasov equation 
	\begin{equation}
		\partial_t f_t + \dive_x(vf_t) + \dive_v((\nabla_x U*f_t) f_t) = 0.
	\end{equation}

	In other words, solutions to the standard Vlasov equation are minimizers to the action functional $F$ as above.
	
	In Figure \ref{fig:pairwiseinteracting1} we plot the optimal paths of four particles with given initial and final positions in the case $ \psi(v)= \frac{1}{2}|v|^2 $ and $ U(x)= 50|x|^2 $.

  \end{example}
  
  \begin{example}[Pairwise interacting particles]\label{example: pairwise}
	If
	\[
		\Phi(f) =  \langle f, \psi \rangle + \frac12\langle f, U*f\rangle,
		\]
		and both $\psi,\, U\in C(\R^d_x \times \R^d_v)$ do depend on $v$, the situation is significantly more complex. We can find sufficient conditions for (A1) and (A2) to hold:

		Define the pairwise interaction potential $\psi_2:\R^d_x \times \R^d_v \times \R^d_{x'} \times \R^d_{v'}\to \R$ as
		\begin{equation}\label{eq: pairwise potential}
			\psi_2(x,v,x',v') := \frac12 \psi(x,v) + \frac12\psi(x',v') + \frac12 U(x-x',v-v').		
		\end{equation}
		Then we can formulate growth and continuity conditions, where again $c,C>0$ are constants:
		\begin{itemize}
			\item [\textbf{(B1)}]$\displaystyle{-C + c(|v|^p + |v'|^p) \leq \psi_2(x,v,x',v') \leq C + C(|v|^p + |v'|^p)}$ \quad for all $x,v,x',v'$.
			\vspace*{0.2cm}
			\item [\textbf{(B2)}]$\displaystyle{
			|\nabla \psi_2(x,v,x',v')| \leq C + C|v|^{p-1} + |v'|^{p-1}	}$ \quad for all $x,v,x',v'$.
		\end{itemize}
		Then since $\Phi(f) = \langle f \otimes f, \psi_2 \rangle$, we have that (B1) implies (A1) and (B2) implies (A2). If $\psi_2$ is convex in $(v,v')$, then $\Phi = \Phi^\rel$. In fact, if $\pi_i$ are Markov kernels and $\sum_{i=1}^N \lambda_i \pi \in MK_p(\TM)$ as in \eqref{eq:Relaxation}, then
		\begin{align*}
			\sum_{i=1}^N \lambda_i \langle f\pi_i \otimes f\pi_i, \psi_2 \rangle = \left\langle (f\otimes f)\left(\sum_{i=1}^N \lambda_i\pi_i \otimes \pi_i\right),\psi_2 \right\rangle
			\geq \langle f\otimes f, \psi_2 \rangle,
		\end{align*}
		since $\sum_{i=1}^N \lambda_i\pi_i \otimes \pi_i:\R^d_x \times \R^d_v \times \R^d_{x'} \times \R^d_{v'} \to \PM(\R^d_v \times \R^d_{v'})$ is a martingale kernel and $\psi_2$ is convex in $(v,v')$.
\begin{itemize}
\item Convexity of $\psi_2$ in $ (v,v') $ implies relaxedness of $\Phi$. Whether the converse holds is unclear to us.
\item Convexity of $\psi_2$ in $ (v,v') $ implies convexity of $\psi$ by taking $v=v'$. If $U$ and $\psi$ are both convex in $ v $, then so is $\psi_2$.
\item Taking $\psi(x,v) = |v|^2$ and $U(x,v) = \alpha|v|^2$, we see that $\psi_2$ is convex and (B1),(B2) are satisfied for $p=2$ whenever $\alpha > -1$.
\item For $\alpha=-1$ we have $\Phi(f) = \langle f \otimes f, v\cdot v' \rangle = \langle f,v\rangle^2$, which no longer satisfies (A1).
\item In general, the relaxation of $\Phi(f) = \langle f\otimes f, \psi_2 \rangle$ ceases to be a quadratic form. Take e.g. $\psi(x,v) = \dist(v,\{\pm1\})$ the two-well potential, $U(x,v) = \alpha \dist(v,\{0,\pm 4\})$. One can show that for $ \alpha $ sufficiently large and $ p\in [0,\varepsilon]\cup[1-\varepsilon,1] $, $ \varepsilon>0 $ sufficiently small, that
\[
\Phi^\rel((1-p)\delta_{(0,-1)} + p \delta_{(0,1)}) = \min(p,1-p).
\]
This shows that $\Phi^\rel$ cannot be a quadratic functional of the type $\Phi^\rel(f) = \langle f\otimes f, \psi_2 \rangle$. To prove the above formula one first shows that for $ p $ sufficiently small no mass in $ (0,-1) $ is spread by a martingale kernel in the relaxation. Then, one can show that the best choice is a martingale kernel distributing the mass in $ (0,1) $ to the points $ (0,-1) $ and $ (0,3) $ with equal probability. This is due to the shape of the interaction potential $ U $ and allows to cancel the interaction of the particles in $ (0,-1) $ and $ (0,3) $, see Figure \ref{fig:Example2.12}. The case $ p\in [1-\varepsilon,1] $ is symmetric to the situation $ p\in [0,\varepsilon] $.
\end{itemize}
\begin{figure}[!ht]
	\centering
	\includegraphics[width=0.7\linewidth]{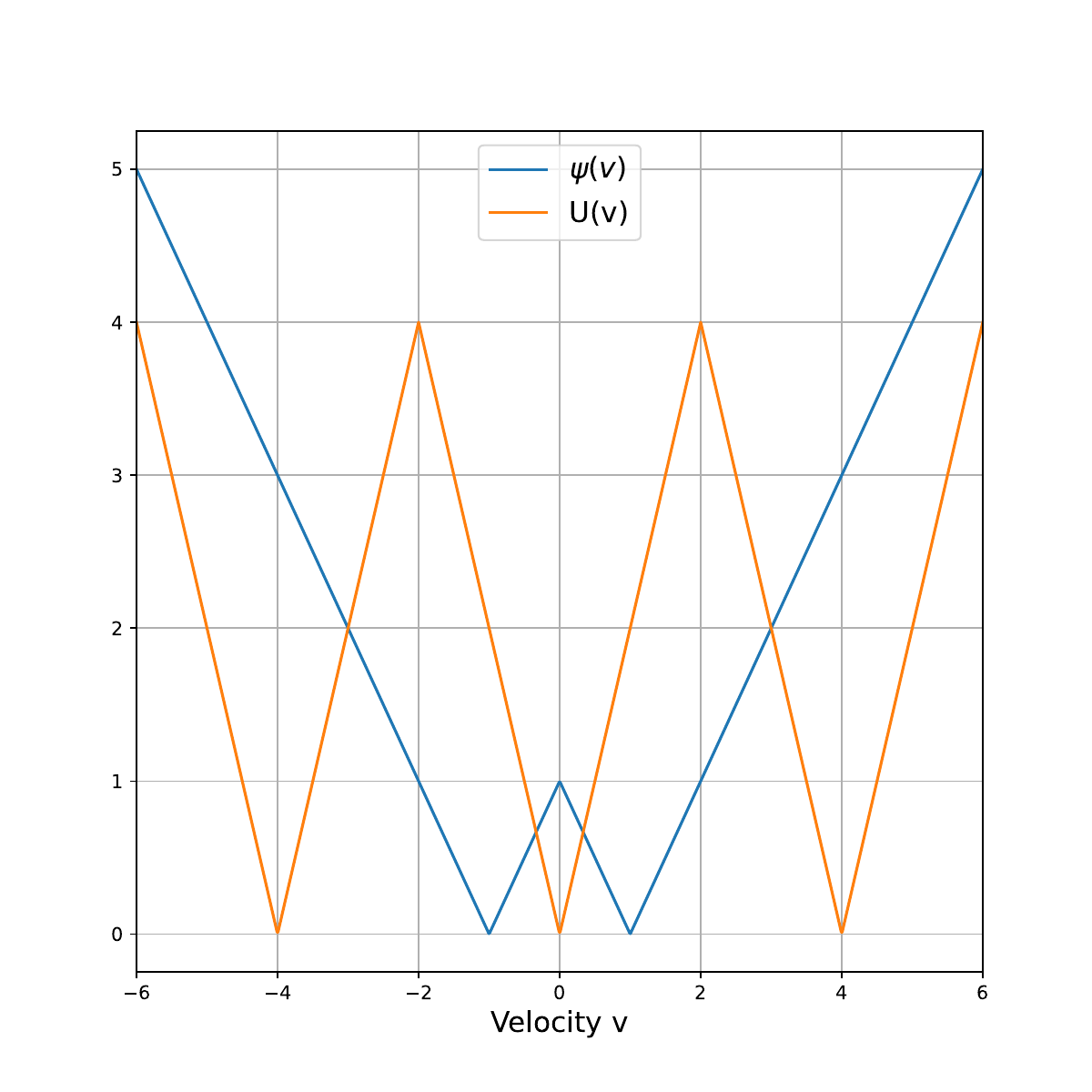}
	\caption[]{Plot of the potential $\psi(v) = \dist(v,\{\pm1\})$ and interaction potential $ U(v)=\alpha \dist(v,\{0,\pm 4\}) $, $ \alpha=2 $, in Example \ref{example: pairwise}.}
	\label{fig:Example2.12}
\end{figure}
Finally, in Figure \ref{fig:pairwiseinteracting2} we plot the optimal paths of four particles interacting via potentials $ \psi(v)= \frac{1}{2}|v|^2 $ and $ U(x,v) =50(|v|^2-10) \exp(-|x|^2) $ when the initial and final positions are prescribed.

\end{example}
  
  \begin{example}[Entropy regularization]
   For 
   \begin{equation}
  \Phi(f):= \int_{\R^n} \log(f_v)f_v(dv) + \dualbra{f}{\dfrac{1}{2}(|x|^2+|v|^2)},
   \end{equation}
   it turns out that $F$ is not lower semicontinuous. In fact, no minimizer exists for the boundary coupling $\Gamma_b:= \delta_{(0,0)}\in \PM(\M\times \M)$. The measure $P:=\delta_{t\mapsto 0}\in \PM(PS_T^2)$ clearly has infinite action. On the other hand, we have
   \begin{equation}
	  \Phi(f) \geq \Phi\left(\delta_0\otimes \frac1Ze^{-|v|^2/2}\right),
   \end{equation}
   with equality if and only if $f=\delta_0 \otimes  \frac1Ze^{-|v|^2/2}$. However, there is no $P\in \PM(PS_T^2)$ with $P_{t,\dot t} = \delta_0 \otimes  \frac1Ze^{-|v|^2/2}$ for almost every time $t$. However, we find that
   \begin{equation}
	  \Phi^\rel(\delta_{(0,0)}) = \Phi\left(\delta_0\otimes \frac1Ze^{-|v|^2/2}\right),
   \end{equation}
   since $\pi(0,0,dv) :=  \frac1Ze^{-|v|^2/2}\,dv$ is a martingale kernel and $ f\mapsto \Phi(f) $ is convex. Thus the relaxed action has minimizer $P = \delta_{t\mapsto 0}$.
  \end{example}



\section{Solving the Cauchy problem}\label{sec:CauchyGenValsovEq}
In this subsection, we limit ourselves to the specific situation of Example \ref{example: pairwise}, where
\begin{equation}
	\Phi(f)  = \langle f, \psi \rangle + \frac12 \langle f, U\ast f \rangle = \langle f\otimes f, \psi_2 \rangle.
\end{equation}
Throughout this section we assume that $ U $ is symmetric, i.e. $ U(x,v)= U(-x,-v) $. Furthermore, we assume for simplicity $p=2$ and the following strict convexity and regularity, for constants $c,C>0$:
\begin{itemize}
	\item [\textbf{(C1)}]$\displaystyle{
	\nabla_v \nabla_v \psi_2(x,v,x',v') \geq c  \,\id_{\R^d_v}
	}$ \quad for all $x,v,x',v'$.
	\vspace*{0.2cm}
	\item [\textbf{(C2)}]$\displaystyle{
	|\nabla^3 \psi_2(x,v,x',v')|+|\nabla^2 \psi_2(x,v,x',v')| + |\nabla \psi_2(x,0,x',0)|+ |\psi_2(x,0,x',0)| \leq C
	}$ 
	\\
	for all $x,v,x',v'$.
\end{itemize}

Note first that (C1) and (C2) imply (B1) and thus (A1), using the Taylor expansion
\begin{align*}
\psi_2(x,v,x',v') \geq & \psi_2(x,0,x',0) + \nabla_{(v,v')}\psi_2(x,0,x',0)\cdot (v,v') +  \frac{c}{2}|v|^2 + \frac{c}{2}|v'|^2\\
\geq & -C -C(|v|+|v'|) +  \frac{c}{2}|v|^2 + \frac{c}{2}|v'|^2 \\
\geq  &- C-C^2/2 +  \frac{c}{4}|v|^2 + \frac{c}{4}|v'|^2 
\end{align*}
by Young's inequality, and likewise for the upper bound. Similarly (C2) implies (B2) and thus (A2), since
\begin{align*}
|\nabla \psi_2(x,v,x',v')| \leq |\nabla \psi_2(x,0,x',0)| + C\|\nabla^2\psi_2\|_\infty(|v|+|v|') \leq C + C(|v|+|v'|).
\end{align*}
We note that convexity of $\psi_2$ in $(v,v')$, which implies $\Phi^\rel = \Phi$, is a stronger condition than separate convexity of $\psi_2$, i.e. $\nabla_v\nabla_v \psi_2 \geq 0$, which is the non-strict version of (C1).

We now state well-posedness of the Cauchy problem for the generalized Vlasov equation \eqref{eq:GenVlasovEq}.
	\begin{theorem}\label{thm:CauchyProblem}
		Assume that $\Phi(f) = \langle f, \psi \rangle + \frac12 \langle f, U \ast f \rangle = \langle f\otimes f, \psi_2 \rangle$ as in example \ref{example: pairwise}, with $\psi_2$ satisfying (C1), (C2). Define the Lagrangian $L[f]:=\psi + U\ast f$.
		\begin{itemize}
		\item [(i)]		Let $ f\in \PM_2(\R^{2d})$. Then there is a unique probability measure $ P\in \PM(PS_\infty^2)$ with statistics $f_t:=P_{t,\dot t}\in \PM_2(\TM)$ satisfying
		\begin{align*}
			\begin{cases}
				f_0 = f&\\
				\frac{d}{dt}\left[\nabla_v L[f_t](\gamma_t,\dot \gamma_t) \right] - \nabla_x L[f_t](\gamma_t,\dot \gamma_t) = 0&\text{ for almost every }\gamma\in \supp P \subseteq PS_\infty^2.
			\end{cases}
		\end{align*}
		\item [(ii)]
		The statistics $f_t$ are distributional solutions in $C([0,\infty);\PM_2(\TM))$ of the generalized Vlasov equation
		\begin{align*}
			\begin{cases}
				f_0 = f\\
				\partial_t f_t + \dive_x(vf_t) + \dive_v(A[f_t]f_t) = 0\\
			\nabla_v\nabla_v L[f_t] A[f_t] - \nabla_v\nabla_v U * (A[f_t]f_t) = \nabla_x L[f_t] - \nabla_x\nabla_vL[f_t]v + \nabla_x \nabla_v U *(v f_t).
			\end{cases}
		\end{align*}

		\item [(iii)] Let $ f,\, \tilde{f}\in \PM_2(\TM) $, and let $P,\tilde P \in \PM(PS_\infty^2)$ be the respective unique solutions of (i), with statistics $f_t,\tilde f_t$. Then we have for some constant $ C>0 $
		\begin{align*}
			W_1(f_t,\tilde{f}_t)\leq Ce^{Ct}W_1(f,\tilde{f}) \quad \text{ for all }t\geq 0.
		\end{align*}
	\end{itemize}
	\end{theorem}

	We construct solutions using a Banach fixed point argument, similar to the construction by Dobrushin \cite{dobrushin1979vlasov}. More precisely, given $P\in \PM(PS_\infty^2)$ with statistics $f_t$, we solve the characteristic system \eqref{eq:PathELEquation}, yielding curves $\gamma^{xv}\in PS_\infty^2$ for every initial value $(x,v)\in \TM$, and define $ \tilde P:= E_f[\delta_{\gamma^{xv}}] \in \PM(PS_\infty^2)$. This constitutes a map  $P\mapsto \tilde P$ for which we seek a fixed point. 
	
	We first prove the stability of the characteristic system \eqref{eq:PathELEquation}.
	\begin{lemma}\label{lem:CharactersitcSystemEstimates1}
		Assume that $\psi,\, U\in C^3(\TM)$ satisfy (C1), (C2). Define for $f\in \PM_2(\TM)$ the Lagrangian $L[f] = \psi + U\ast f\in C^3(\TM)$. Then
\begin{itemize}
	\item [(i)] The Lagrangian satisfies $c\,\id_{\R^d_v} \leq \nabla_v \nabla_v L[f](x,v)$,
	\begin{equation}
		|\nabla^3 L[f](x,v)| + |\nabla^2 L[f](x,v)| + |\nabla L[f](x,0)| + |L[f](x,0)|\leq C,
	\end{equation}
	and
	\begin{equation}
		\begin{aligned}
		&|\nabla^2 L[f](x,v) - \nabla^2 L[\tilde f](x,v)|+ |\nabla L[f](x,v) - \nabla L[\tilde f](x,v)|\leq  CW_1(f,\tilde f)
		\end{aligned}
	\end{equation}
	for all $(x,v)\in \R^d_x\times \R^d_v $ and all $f,\tilde f \in \PM_2(\TM)$.
	\item [(ii)] Let $f\in C([0,\infty);\PM_2(\TM))$ be a curve of statistics. For $(x,v)\in \TM$, there exists a unique solution $\gamma^{xv}\in C^1([0,\infty);\R^d_x)\subseteq PS_\infty^2$ to the characteristic system
	\begin{equation}\label{eq: characteristic}
		\begin{cases}
		(\gamma^{xv}_0, \dot \gamma^{xv}_0)  = (x,v)&\\
		\nabla_v L[f_t](\gamma^{xv}_t,\dot \gamma^{xv}_t) = \int_0^t \nabla_x L[f_s](\gamma^{xv}_s,\dot \gamma^{xv}_s)\,ds + \nabla_v L[f_0](x,v)&\text{ for all }t\geq 0,
		\end{cases}
	\end{equation}
which satisfies $|\gamma^{xv}_t| + |\dot \gamma^{xv}_t| \leq  Ce^{Ct}(1+|x|+|v|)$ and
\begin{equation}\label{eq: different points}
	|\gamma^{xv}_t - \gamma^{x'v'}_t| + |\dot \gamma^{xv}_t - \dot \gamma^{x'v'}_t| \leq  Ce^{Ct}(|x-x'|+|v-v'|).
\end{equation}
\item [(iii)] Let $f,\tilde f\in C([0,\infty);\PM_2(\TM))$, and $\gamma^{xv},\tilde \gamma^{xv}\in C^1([0,\infty);\R^d_x)$ the respective characteristic solutions to \eqref{eq: characteristic}. Then for any $(x,v)\in \TM$, $ t\geq0 $ we have
\begin{equation}\label{eq: different f}
|\gamma^{xv}_t - \tilde\gamma^{xv}_t| + |\dot \gamma^{xv}_t - \dot{\tilde{\gamma}}^{xv}_t| \leq C(e^{Ct}-1)\max_{s\in [0,t]}W_1(f_s,\tilde f_s).
\end{equation}
\end{itemize}
	\end{lemma}

	\begin{proof}
\emph{Proof of (i):} We make note of the following link between $\psi_2$ and $L[f]$:
\[
L[f](x,v) = \psi(x,v) + U\ast f(x,v) = \int_{\R^d_{x'}\times \R^d_{v'}} (2\psi_2(x,v,x',v') - \psi(x',v'))f(dx',dv').	
\]
All bounds on $\nabla^3 L[f],\nabla^2 L[f], \nabla L[f]$ follow immediately from (C1) and (C2), since $f$ is a probability measure. The bound on $|L[f](x,0)|$ follows since $\psi(x,v) = \psi_2(x,v,x,v) + \frac12 U(0,0)$.

To show the stability of $f\mapsto \nabla L[f]$ and $f\mapsto \nabla^2L[f]$, choose a coupling $\Gamma\in \PM_2(\R^d_x\times \R^d_v \times \R^d_{x'}\times \R^d_{v'})$ between $f$ and $\tilde f$ with $W_1(f,\tilde f) =\langle \Gamma, |x-x'|+|v-v'|\rangle$. Then
\[
	\begin{aligned}
|\nabla L[f](x_0,v_0) - \nabla L[\tilde f](x_0,v_0)|  =& |\nabla U \ast (f-\tilde f)(x_0,v_0)|\\
 \leq& |\langle\Gamma, \nabla U(x_0 - x, v_0 -v) - \nabla U(x_0-x',v_0-v')\rangle|\\
  \leq& \|\nabla^2 U \|_\infty W_1(f,\tilde f),	
	\end{aligned}
\]
and likewise
\[
	|\nabla^2 L[f](x_0,v_0) - \nabla^2 L[\tilde f](x_0,v_0)| \leq \|\nabla^3 U \|_\infty W_1(f,\tilde f).
\]
\emph{Proof of (ii):} We employ the change of variables $\phi_f:\TM \to \R^d_x \times \R^d_p$ defined through
\[
\phi_f(x,v) = (x,\nabla_v L[f](x,v)).	
\]

Clearly $|\nabla \phi_f(x,v)|\leq C$. Since $\nabla_v\nabla_v L[f]\geq c\,\id_{\R^d_v}$, $\phi_f$ is a Bilipschitz diffeomorphism, with $|\nabla \phi_f^{-1}|\leq C$. Existence and stability of solutions to the characteristic system is achieved by solving the Hamiltonian system $(\gamma_t^{xv},p_t^{xv}):[0,\infty)\to \R^d_x\times \R^d_p$,
\[
\begin{cases}
	(\gamma^{xv}_0,p^{xv}_0) = \phi_{f_0}(x,v)\\
	\dot \gamma^{xv}_t = \phi_{f_t}^{-1}(\gamma^{xv}_t,p^{xv}_t)_v\\
	\dot p^{xv}_t = \nabla_x L_{f_t}\circ \phi_{f_t}^{-1}(\gamma^{xv}_t,p^{xv}_t).
\end{cases}
\]

Since the right-hand side is $C$-Lipschitz, well-posedness of this first-order system follows.

\emph{Proof of (iii):} After the changes of variables, we estimate
\[
	\frac{d}{dt}(|\gamma^{xv}_t - \tilde \gamma^{xv}_t| + |p^{xv}_t - \tilde p^{xv}_t|) \leq C(|\gamma^{xv}_t - \tilde \gamma^{xv}_t| + |p^{xv}_t - \tilde p^{xv}_t| + W_1(f_t,\tilde f_t)).
\]

By Gronwall's inequality
\[
	|\gamma^{xv}_t - \tilde \gamma^{xv}_t| + |p^{xv}_t - \tilde p^{xv}_t| \leq (e^{Ct}-1)\max_{s\in[0,t]}W_1(f_s,\tilde f_s),
\]
and (iii) follows by inverting the changes of variables.
\end{proof}

We can now prove Theorem \ref{thm:CauchyProblem}:
\begin{proof}[Proof of Theorem \ref{thm:CauchyProblem}]\emph{Proof of (i):} We work on the space
	\begin{align*}
		X := C^0(\TM;C^1([0,T];\R^d_x))
	\end{align*} 
for $T>0$. On this space, define the pseudodistance $d: X \times X \to [0,\infty]$ through
\[
d(\gamma,\tilde \gamma) := \sup_{t\in[0,T],\, x\in \R^d_x,\, v\in \R^d_v}\dfrac{|\gamma^{xv}_t-\tilde \gamma^{xv}_t| + |\dot \gamma^{xv}_t- \dot{\tilde \gamma}^{xv}_t|}{1+|x|+|v|}.	
\]

We define the map $F: X \to X$ for which we seek a fixed point. First, given $\gamma\in X$, define the induced statistics $f_t^\gamma:=E_f[\delta_{(\gamma^{xv}_t,\dot\gamma^{xv}_t)}]\in \PM_2(\TM)$ for $t\in[0,T]$. Then, we define $(F\gamma)^{xv}\in C^1([0,T];\R^d_x)$ as the unique solution to \eqref{eq: characteristic} with $f_t = f_t^\gamma$.

We note that $\max_{s\in[0,T]}W_1(f_s^\gamma,f_s^{\tilde \gamma}) \leq d(\gamma,\tilde \gamma)$. By Lemma \ref{lem:CharactersitcSystemEstimates1} (iii) it then follows that $d(F\gamma,F\tilde \gamma) \leq CTe^{CT}d(\gamma,\tilde \gamma)$. Furthermore, $F$ to the complete metric space $X_0 := \{\gamma\in X\,:\,d(\gamma, 0)<\infty\}$, $F$ becomes a contraction if $CTe^{CT}<1$, and thus has a unique fixed point $\gamma\in X_0$. Taking $P = E_f[\delta_{\gamma^{xv}}]$, we get $P_{t,\dot t} = f^\gamma_t$, which implies (i) on $[0,T]$. Repeating the construction at time $T$, we can extend $X$ and thus $P$ to $T=\infty$, proving (i). Statement (ii) follows immediately from Remark \ref{rem: Vlasov}.

\emph{Proof of (iii)}: Let $\gamma,\tilde \gamma \in C^0(\TM;C^1([0,\infty];\R^d_x))$ be the respective solutions for $f,\tilde f$. Combining estimates \eqref{eq: different points} and \eqref{eq: different f} yields for any $T>0$ and any $x,v,x',v'$
\[
	\max_{t\in[0,T]}|\gamma^{xv}_t - \tilde\gamma^{x'v'}_t| + |\dot \gamma^{xv}_t - \dot{\tilde \gamma}^{x'v'}_t| \leq  Ce^{CT}(|x-x'|+|v-v'|) + C(e^{CT}-1)W_1(f,\tilde f).
\]

Choose $T$ small enough that $C(e^{CT}-1)\leq \frac12$. Choose an optimal coupling $\pi\in \PM_1(\R^d_x\times \R^d_v\times \R^d_{x'}\times \R^d_{v'})$ between $f$ and $\tilde f$, inducing couplings $\pi_t = (\gamma^{xv}_t,\dot \gamma^{xv}_t,\tilde \gamma^{x'v'}_t,\dot{\tilde \gamma}^{x'v'}_t)_\#\pi$ between $f_t$ and $\tilde f_t$, so that
\[
\max_{t\in[0,T]}W_1(f_t,\tilde f_t)	\leq Ce^{CT}W_1(f,\tilde f) + \frac12 \max_{t\in[0,T]}W_1(f_t,\tilde f_t).
\]

Iterating the above and increasing $C$ yields
\[
	W_1(f_t,\tilde f_t) \leq C e^{Ct}W_1(f,\tilde f)
	\]
for all times $t\in[0,\infty)$.
	\end{proof}

\section{The relaxed action}\label{sec:PropertiesRelaxFunctional}
In this section we give alternative representations of the relaxed action in Theorem \ref{theorem: relaxation} and prove several properties used later on.  
 
\subsection{Alternative representations of the relaxed action}
We now show that the relaxation \eqref{eq:Relaxation} can be expressed in two alternative ways. The first is through replacing finite convex combinations with uncountable infinite convex combinations, represented through an integral. The second eschews Markov kernels altogether, instead writing the relaxation through action minimization on the tangent space, where particles are only allowed to move through the tangent space $\R^d_v$ while keeping their positions constant.

\begin{lemma}\label{lem:AltFormRelaxation}
	The functional $ \Phi^\rel $ defined by \eqref{eq:Relaxation} satisfies
	\begin{align}\label{eq:AltFormRelaxation}
		\Phi^\rel(f)= \inf\left\lbrace \int_0^1\Phi(f\pi_s)\, ds \,|\,  \pi:\TM\times [0,1]\to \PM_p(\R^d_v), \right.
		\\ 
		\left. \int_0^1\pi_s\, ds \in MK_p(\TM) \right\rbrace 
	\end{align}
	and
	\begin{align}\label{eq: path relaxation}
		\Phi^\rel(f) = \inf\left\lbrace \int_0^1 \Phi(R_{s,\dot s})\,ds\,|\, R\in \PM(TPS^p), R_1 = f \right\rbrace,
	\end{align}
	where $TPS^p = \{\gamma\in W^{1,p}([0,1];\R^d_x\times \R^d_v)\,|\, \gamma_x\text{ constant},\gamma_v(0) = 0\}$ denotes vertical paths in the tangent bundle. For $R\in \PM(TPS^p)$ and almost every $s\in [0,1]$, $R_{s,\dot s}\in \PM(\R^d_x\times \R^d_v)$ is defined as $\E_R[\delta_{(\gamma(s)_x,\dot \gamma(s)_v)}]$ similarly as $ P_{t,\dot t} $ in Section \ref{subsec:Notation}.
\end{lemma}
\begin{proof}
	\emph{Proof of \eqref{eq:AltFormRelaxation}.} Let us write $ \Psi $ for the functional defined via formula \eqref{eq:AltFormRelaxation}. Certainly, we have $ \Phi^\rel(f)\geq \Psi(f) $. For the other inequality, let $ \eps>0 $ and $ \pi:\TM\times [0,1]\to \PM_p(\R^d_v) $ such that $ \int_0^1\pi_s\, ds \in MK_p(\TM) $ and
	\begin{align*}
		\int_0^1\Phi(f\pi_s)\, ds \leq \Psi(f) +\eps < \infty.
	\end{align*}
	We define a mollification $ \pi^k_s $ in the variable $ s\in [0,1] $ via a standard mollifier $ \eta_k:\R\to \R $, i.e. $ \eta_k(r):=k \eta(rk) $, $ \eta\in C_c^\infty(\R;\R) $, $ \supp \eta \subset (-1,1) $, $ \eta\geq 0 $, $ \int_{\R}\eta(r)\, dr=1 $. More precisely, we extend $ s\mapsto \pi_s $ periodically to all of $ \R $ and define the mollified Markov kernel $\pi^k_s: \TM \times [0,1] \to \PM_p(\R^d_v)$,
	\begin{align*}
		\pi^k_s(x,v) := & \int_\R \eta_k(s-r) \pi_r(x,v)\, dr .
	\end{align*} 
	We observe first that $\pi^k_s$ is an admissible Markov kernel, since
	\[
	\int_0^1 \pi^k_s\,ds = \int_0^1 \pi_s\,ds \in MK_p(\TM).	
	\]
	
	As $k\to \infty$, the mollified kernels converge to the original in the sense that
	\begin{equation}
		\lim_{k\to\infty} \int_0^1 \int_{\TM} W_p^p(\pi_s(x,v),\pi^k_s(x,v))f(dx,dv)\,ds = 0.
	\end{equation}
	This is shown by first using Lusin's theorem to approximate $\pi_s$ with a kernel that is continuous in $s$, mollifying the continuous kernel, and using the triangle inequality for $W_p$.
	
	The next step involves discretizing the mollified kernel, yielding for $N\in\N$, $i=1,\ldots,N$, the Markov kernels $\pi_i^{k,N}:\TM \to \PM_p(\R^d_v)$,
	\begin{equation}
		\pi_i^{k,N}(x,v) := \frac1N \int_{\frac{i-1}{N}}^{\frac{i}{N}} \pi^k_s(x,v)\,ds.
	\end{equation}
	Since every $\pi^k_s$ is continuous in $s$, we infer
	\begin{equation}
		\lim_{N\to\infty} \int_0^1 \int_{\TM} W_p^p(\pi^k_s(x,v),\pi^{k,N}_{i(s,N)}(x,v))f(dx,dv)\,ds = 0,
	\end{equation}
	where $i(s,N) := \lceil sN \rceil$. Again, we have $\frac{1}{N}\sum_{i=1}^N \pi_i^{k,N}\in MK_p(\TM)$. Using the continuity property (A2) of $\Phi$, we estimate
	\begin{align*}
		\frac1N \sum_{i=1}^N \Phi(f\pi^{k,N}) &\leq \int_0^1 \Phi(f\pi_s)\,ds 
		\\
		&\quad + \int_0^1 \langle f\pi_s + f\pi^{k,N}_{i(s,N)}, C + C|v|^p \rangle ^{(p-1)/p} W_p\left(f\pi_s,f\pi^{k,N}_{i(s,N)}\right)\,ds.
	\end{align*}
	
	Using H\"older's inequality and the convexity of the Wasserstein distance allows us to estimate the error term by
	\begin{equation}
		\begin{aligned}
			&		\int_0^1 \langle f\pi_s + f\pi^{k,N}_{i(s,N)}, C + C|v|^p \rangle ^{(p-1)/p} W_p\left(f\pi_s,f\pi^{k,N}_{i(s,N)}\right)\,ds\\
			\leq & \left(\int_0^1 \langle 2f\pi_s,C+C|v|^p \rangle\,ds\right)^{(p-1)/p}\left(\int_0^1 W_p^p(f\pi_s,f\pi^{k,N}_{i(s,N)})\,ds\right)^{1/p}.
		\end{aligned}
	\end{equation}
	The second integral tends to zero along a diagonal sequence, while the first term is finite by the growth condition (A1), showing that $\Phi^{\rel} \leq \Psi$ and consequently \eqref{eq:AltFormRelaxation}.
	
	\emph{Proof of \eqref{eq: path relaxation}.} Let us write $\Xi$ for the functional defined via formula \eqref{eq: path relaxation}. We first show that $\Xi(f)\leq \Phi^\rel(f)$:
	
	Given $N\in \N$ Markov kernels $\pi_i:\R^d_x \times \R^d_v \to \PM(\R^d_v)$ and weights $\lambda_i\in [0,1]$ with $\sum_{i=1}^N \lambda_i=1$ and $\sum_{i=1}^N\lambda_i \pi_i$ a martingale kernel, we define for $i=0,\ldots,N$ the intermediate times $t_i:=\sum_{j=1}^i \lambda_i$ and show the existence of a probability measure $R\in \PM(TPS^p)$ with $R_{s,\dot s} = f\pi_i$ for $s\in (t_i,t_{i+1})$ and $R_1 = f$. We do so by defining $N$ random variables $X_i:\R^d_x\times \R^d_v \times \Omega \to \R^d_v$ with distribution $X_i(x_0,v_0,\cdot) \sim \pi_i(x_0,v_0)$. In fact, we can choose our probability space as the circle $\Omega = \T^1 = \R/\Z$, $\mathcal{F} = \mathcal{B}(\T^1)$, $P = \Lm^1$ the Haar measure on the circle. We do not need the $X_i:\R^d_x\times \R^d_v \times \T^1\to \R^d_v$ to be independent. With this, we define for every $(x_0, v_0,\theta)\in \R^d_x\times \R^d_v\times \T^1$ the curve $\gamma(x_0,v_0,\theta)\in TPS^p$ through $\gamma(x_0,v_0,\theta)_x =x_0$ and
	\[
	\dot \gamma(x_0,v_0,\theta,s)_v = X_i\left(x_0,v_0,\frac{s-t_i}{t_{i+1}-t_i}+\theta\right) \text{ for }s\in(t_i,t_{i+1}).	
	\]
	Note that for all $x_0,\in\R^d_x$, $v_0\in \R^d_v$, $\theta_0\in \T^1$ and all $i=0,\ldots,N_1$ we have by a change of variables
	\[
	\int_{t_i}^{t_{i+1}} \dot \gamma(x_0,v_0,\theta_0,s)\,ds = \lambda_i \int_{\T^1} X_i(x_0,v_0,\theta)\,d\theta = \lambda_i\langle \pi_i,v \rangle,
	\]
	so that
	\[
	\gamma(x_0,v_0,\theta_0,1) = \int_0^1 \dot \gamma(x_0,v_0,\theta_0,s)_v\,ds = \sum_{i=1}^N \lambda_i \langle \pi_i,v\rangle = v_0
	\]
	by the martingale property of $\sum \lambda_i \pi_i$. We define $R\in \PM(TPS^p)$ through
	\[
	R := \int_{\R^d_x\times \R^d_v} \int_{\T^1} \delta_{\gamma(x,v,\theta)}\,d\theta f(dx,dv),
	\]
	so that $R_1 = f$ and $R_{s,\dot s} = f\pi_i$ for $s\in (t_i,t_{i+1}]$. This shows that $\Xi\leq \Phi^\rel$.
	
	Finally, we show that $\Xi\leq \Psi$. Given $R\in \PM_p(TPS^p_T)$ with $R_1 = f\in \PM_p(\R^d_x \times \R^d_v)$, define for $f \otimes \Lm^1$-almost every $(x_0,v_0,s)\in   \R^d_x \times \R^d_v \times [0,1]$ the Markov kernel $\pi_s(x_0,v_0)$ as the conditional expectation (which exists by the disintegration theorem \cite[Theorem 5.3.1]{ambrosio2005gradient})
	\[
	\pi_s(x_0,v_0) := \E_R[\delta_{\dot \gamma(s)_v}\,|\,\gamma(1) = (x_0,v_0)]	\in \PM_p(\R^d_v).
	\]
	
	Then by the tower property \cite[Theorem 34.2]{billingsley2017probability} we have $f\pi_s = \E_R[\delta_{(\gamma(s)_x,\dot \gamma(s)_v)}] = R_{s, \dot s}$ for $\Lm^1$-almost every $s$. We extend $\pi_s$ to all of $[0,1]\times (\R^d_x \times \R^d_v \setminus \mathrm{supp}f)$ by $\pi_s(x_0,v_0):= \delta_{v_0}$. Then $\int_0^1 \pi_s\,ds$ is clearly a martingale kernel for such $(x_0,v_0)$. For $(x_0,v_0)\in \mathrm{supp}f$ we have by Fubini's theorem \cite[Theorem 18.3]{billingsley2017probability} that
	\[
	\int_0^1 \langle \pi_s(x_0,v_0), v \rangle \, ds = \E_R\left[\int_0^1 \langle \delta_{\dot \gamma(s)_v},v \rangle \, ds \,|\,\gamma(1) = (x_0,v_0)\right] = v_0.
	\]
	
	This is precisely the martingale property, concluding the proof.
\end{proof}

\subsection{Properties of the relaxed action}
The following lemma summarizes useful properties of $ \Phi^\rel $.
\begin{prop}\label{prop:PropertiesRelaxation}
Let $\Phi$ satisfy assumptions (A1) and (A2). Then
\begin{enumerate}
	\item There are constants $0<c<C$ such that $\langle f, -C+c |v|^p \rangle \leq \Phi^\rel(f) \leq \langle f, C+C |v|^p \rangle$.
	\item $\Phi^\rel$ is $W_p$-continuous with 
		\begin{equation*}
			|\Phi^\rel(f) - \Phi^\rel(f')| \leq \langle f+f',  C+C |v|^p \rangle^{(p-1)/p} W_p(f,f').
		\end{equation*}
	
	\item $\Phi^\rel$ is increasing in the convex order.
	
	\item It holds $ (\Phi^\mathrm{conv})^\inc\leq (\Phi^\inc)^\mathrm{conv} $ and $(\Phi^\mathrm{conv})^\inc \leq \Phi^\rel \leq \Phi^\inc $. In particular, if $\Phi$ is convex, then $\Phi^\rel = \Phi^\inc$ and $\Phi^\rel$ is also convex.
\end{enumerate}
\end{prop}

Before we prove the above proposition, let us mention the following negative results:
\begin{enumerate}
	\item[(5)] The function $ \Phi^\rel $ is in general not convex.
	\item[(6)] The inequality $ \Phi^\rel \leq \Phi^\inc $ in (4) is in general strict.
	\item[(7)] The function $ (\Phi^\inc)^\mathrm{conv} $ is in general not increasing, more precisely $ (\Phi^\inc)^\mathrm{conv} \neq ((\Phi^\inc)^\mathrm{conv})^\inc $. In particular, in general the inequality $ (\Phi^\mathrm{conv})^\inc\leq (\Phi^\inc)^\mathrm{conv} $ in (4) is strict.
	\item[(8)] Neither $ \Phi^\rel \leq (\Phi^\inc)^\mathrm{conv} $ nor $ (\Phi^\inc)^\mathrm{conv}\leq \Phi^\rel $ holds in general.
\end{enumerate}
We first give counterexamples to (5)-(8).
\begin{example}\label{ex:NotConv}
	Concerning (5) and (6) above we consider the following example. Let $d=1$ and $\Phi(f):= \langle f, \varphi + v^2 \rangle - \langle f, v \rangle^2$, for $\varphi(x,v) := \frac14(v^2-1)^2 \in C^2(\R)$. The function $\Phi$ clearly satisfies the growth and continuity conditions (A1) and (A2). A direct calculation shows that $\Phi^\rel(\delta_0 \otimes \delta_{\pm 1}) = 0$, since $\varphi(\pm 1) = 0$. On the other hand, we find that
	\[
	\begin{aligned}
		  & \Phi^\rel\left(\frac{1}{2}(\delta_0 \otimes \delta_1 + \delta_0 \otimes \delta_{-1})\right)
		\geq \inf\left\{ \int_0^1 \mathrm{Var}\left(\frac12(\pi_s(1) + \pi_s(-1))\right)\,ds \right\} 
		\\
		&\quad = \inf\left\{ \int_0^1 \frac12(\mathrm{Var}(\pi_s(1))+\mathrm{Var}(\pi_s(-1))) +  \dualbra{\frac{\pi_s(1) - \pi_s(-1)}{2}}{v}^2  \right\}
		\\
		&\quad \geq \inf \left\{ \int_0^1   \dualbra{\frac{\pi_s(1) - \pi_s(-1)}{2}}{v}^2 \, ds  \right\}
		\geq \inf\left\{\frac14 \left \langle \int_0^1\pi_s(1) - \pi_s(-1)\,ds, v \right \rangle^2 \right\}
		\\
		& \quad = 1,
	\end{aligned}	
	\]
	where we used the alternative characterization of $\Phi^\rel$ in Lemma \ref{lem:AltFormRelaxation}, the non-negativity of $\varphi$, the parallelogram identity for the variances, and the martingale property
	\begin{align*}
		\left\langle \int_0^1\pi_s(\pm 1)\,ds, v \right\rangle = \pm 1.
	\end{align*}
	Putting in $\pi_s(\pm 1) = \delta_{\pm 1}$ reveals that $\Phi^\rel(g) = 1$, $ g:=\frac{1}{2}(\delta_0 \otimes \delta_1 + \delta_0 \otimes \delta_{-1}) $, in particular that $\Phi^\rel$ is not convex. This yields (5) above.

	The same example also shows that in general $\Phi^\rel < \Phi^\inc$, cf. (6) above: $\Phi$ is increasing in the convex order since $v\mapsto \varphi(v)+|v|^2$ is convex. In particular, $\Phi^\inc(\delta_0 \otimes \delta_0)) = \varphi(0) = 1$ but $\Phi^\rel(\delta_0 \otimes \delta_0) \leq \frac{1}{2}(\Phi(\delta_0\otimes \delta_1) + \Phi(\delta_0\otimes \delta_{-1})) = 0$.
	
	Finally, we also obtain a counterexample for the first part of (8). We have above $\Phi^\rel(g) = 1$, $ g:=\frac{1}{2}(\delta_0 \otimes \delta_1 + \delta_0 \otimes \delta_{-1}) $. However, $ \Phi $ is increasing in the convex order and
	\begin{align*}
		\Phi^{\conv}(g) \leq \frac{1}{2}(\Phi(\delta_0\otimes \delta_1) + \Phi(\delta_0\otimes \delta_{-1})) = 0.
	\end{align*}
	Thus, the inequality $ \Phi^\rel \leq (\Phi^\inc)^{\conv} $ cannot hold in general.
\end{example}
\begin{example}
	Concerning (7) and the second part of (8) we consider the following example. Let $ \Phi(f)=\psi(\dualbra{f}{v}) $ with $ \psi(s):=\frac{1}{4}(s^2-1)^2 $. The function $\Phi$ clearly satisfies the growth and continuity conditions (A1) and (A2). It is increasing, i.e. $ \Phi=\Phi^\inc $ and we have $ \Phi^{\conv}(\delta_0 \otimes \delta_0)=\psi(0)=1 $. However, we observe that
	\begin{align*}
		\pi_1 = \delta_0 \otimes \delta_1, \quad \pi_{-1}=\delta_0 \otimes \delta_{-1}, \quad g:=\dfrac{1}{2}\left( \delta_0 \otimes \delta_1+\delta_0 \otimes \delta_{-1} \right) \succ \delta_0 \otimes \delta_0.
	\end{align*}
	We conclude that
	\begin{align*}
		0\leq \Phi^{\conv}(g) \leq \dfrac{1}{2}\left( \Phi^{\conv}(\delta_0 \otimes \delta_1) + \Phi^{\conv}(\delta_0 \otimes \delta_{-1})\right) =0 < 1= \Phi^{\conv}(\delta_0 \otimes \delta_0).
	\end{align*}
	Hence, $ \Phi^{\conv} $ is not increasing, yielding (7). 
	
	Finally, for the second part of (8) we observe that the latter counterexample yields similarly
	\begin{align*}
		0\leq \Phi^{\rel}(\delta_0 \otimes \delta_0) \leq  \dfrac{1}{2}\left( \Phi(\delta_0 \otimes \delta_1) + \Phi(\delta_0 \otimes \delta_{-1})\right) =0 < 1= \Phi^{\conv}(\delta_0 \otimes \delta_0).
	\end{align*}
	Thus, $ (\Phi^{\inc})^{\conv}\leq\Phi^\rel $ cannot hold in general. 
\end{example}

	\begin{proof}[Proof of Proposition \ref{prop:PropertiesRelaxation}]
		(1): To show the growth condition, choose an admissible kernel $\pi:\TM\times (0,1) \to \PM_p(\R^d_v)$ and estimate
		\[
		\int_0^1 \Phi(f\pi_s)\,ds \geq \int_0^1 \langle -C +c |v|^p, f\pi_s \rangle\,ds = \langle -C + c|v|^p,f \rangle,	
		\] 
		where we used the martingale property of $\pi$. Taking the infimum over all such $\pi$ yields $\Phi^\rel(f) \geq \langle-C +c |v|^p,f\rangle$. The upper bound follows from $\Phi^\rel \leq \Phi$.

		(2): To show the continuity, we use the \emph{coupling method}:

		Let $f,f'\in \PM_p(\TM)$. Choose an admissible Markov kernel $\pi:\TM\times (0,1) \to \PM_p(\R^d_v)$ with 
		\[
			\Phi^\rel(f)\leq \int_0^1 \Phi(f\pi_s)\,ds + \eps.
			\]

			Choose a coupling $\Gamma \in \PM_p(\R^d_x\times \R^d_v \times \R^d_{x'}\times \R^d_{v'})$ of $f$ and $f'$ with
			\[
				\langle \Gamma,|x-x'|^p + |v-v'|^p \rangle = W_p^p(f,f').
			\]

			We construct an admissible Markov kernel $\pi':\supp f'\times (0,1) \to \PM_p(\R^d_v)$: First, we find the Markov kernel $\Gamma':\supp f'\to \PM_p(\TM)$ obtained through disintegration of $\Gamma$ with respect to $f'$, i.e. $\Gamma'(dx,dv,x',v')f'(dx',dv') = \Gamma(dx,dv,dx',dv')$. Then define for $(x',v')\in \supp f'$
			\[
			\pi_s'(x',v'):= \int_{\TM} T_{v'-v}^\# \pi_s(x,v)	\Gamma'(dx,dv,x,v),
			\]
			where $T_{v'-v}:\R^d_v \to \R^d_v$ is the translation $v''\mapsto v''+v'-v$. We have to check that $\pi_s$ is indeed an admissible Markov kernel, i.e. that $\int_0^1 \int_{\R^d_{\tilde v}}\tilde v\,\pi_s'(x',v',d\tilde v) \,ds= v'$ for all $(x',v')\in \TM$.

			By Fubini's theorem
			\begin{align*}
				\int_0^1 \int_{\R^d_{\tilde v}} \tilde v \,\pi_s'(x',v',d\tilde v)\, ds	= & \int_0^1\int_{\R^d_{\tilde v}} \tilde v \int_{\TM} T^\#_{v'-v}\pi_s(x,v,d\tilde v) \Gamma'(dx,dv,x',v') \, ds
				\\
				= &  \int_{\TM} \left( \int_0^1 \int_{\R^d_{\tilde v}} (\tilde v + v' - v) \,\pi_s(x,v,d\tilde v) \,ds\right) \Gamma'(dx,dv,x',v')
				\\
				= & \int_{\TM}v'\,\Gamma'(dx,dv,x',v') = v'\,.
			\end{align*}
			This shows that indeed $\pi_s':\supp f' \to \PM_p(\TM)$ is an admissible kernel. We can extend $\pi_s'$ to $(\TM)\setminus \supp f'$ by taking e.g. $\pi_s'(x,v) = \delta_v$ for $(x,v)\in (\TM)\setminus \supp f'$.

			Now we continue by estimating the $W_p$ distance between $f\pi_s$ and $f'\pi_s'$: construct a coupling $\Gamma_s\in \PM(\TM\times \R^d_{x'}\times \R^d_{v'})$ between $f\pi_s$ and $f'\pi_s'$ through
			\[
			\langle \Gamma_s,\varphi \rangle:= \int_{\TM\times \R^d_{x'}\times \R^d_{v'}}\int_{\R^d_{\tilde v}}\varphi(x,\tilde v,x',\tilde v  + v' -v) \, \pi_s(x,v, d\tilde v)\,\Gamma(dx,dv,dx',dv') 
			\]
		    for $\varphi\in C_b(\TM\times \R^d_{x'}\times \R^d_{v'})$. Disintegrating $\Gamma$ shows that $\Gamma_s$ indeed defines a coupling between $f\pi_s$ and $f'\pi_s'$, allowing us to bound
			\begin{align*}
				&W_p^p(f\pi_s,f'\pi_s')
				\\
				&\leq \langle \Gamma_s, |x-x'|^p+|v-v'|^p \rangle\\
				&= \int_{\TM\times \R^d_{x'}\times \R^d_{v'}} \int_{\R^d_{\tilde v}} \left(  |x-x'|^p + |\tilde v - (\tilde v +v'-v)|^p\right)  \, \pi_s(x,v,d \tilde v)\,\Gamma(dx,dv,dx',dv')\\
				&= \int_{\TM\times \R^d_{x'}\times \R^d_{v'}}  \left(  |x-x'|^p + |v-v'|^p\right)  \, \Gamma(dx,dv,dx',dv')\\
				&= W_p^p(f,f').
			\end{align*}
			In other words, we have a $W_p$ contraction for almost every $s\in (0,1)$. Finally, we use property (A2) to estimate
			\begin{align*}
				\Phi^\rel(f') \leq & \int_0^1 \Phi(f'\pi_s')\,ds\\
				\leq & \int_0^1 \Phi(f\pi_s)\,ds + \int_0^1\langle  f\pi_s + f'\pi_s',  C+C|v|^p\rangle^{(p-1)/p} W_p(f\pi_s, f'\pi_s')\,ds\\
				\leq &\Phi^\rel(f) + \eps + C W_p(f,f')\left(\int_0^1\langle f\pi_s + f'\pi_s', 1+ |v|^p \rangle \,ds\right)^{(p-1)/p}.
			\end{align*}
			Now we simplify the last term:
			\[
			\langle f'\pi_s', |v|^p	\rangle \leq C\langle f\pi_s, |v|^p \rangle + CW_p^p(f\pi_s,f'\pi_s') \leq C \langle f\pi_s + f + f', |v|^p \rangle.
			\]

			We further use the growth conditions (A1) to estimate
			\[
				\begin{aligned}
				\int_0^1 \langle f\pi_s + f'\pi_s', 1 + |v|^p \rangle\,ds &\leq C \int_0^1 \langle f\pi_s + f + f', 1 + |v|^p \rangle \,ds\\
				\leq & C \int_0^1 \Phi(f\pi_s)\,ds + C\langle f + f', 1 + |v|^p \rangle\\
				\leq & C\Phi(f) + C\langle f + f', 1 + |v|^p \rangle\\
				\leq & \langle f + f', C + C|v|^p \rangle,
				\end{aligned}
			\]
			which yields (2).

			(3): Let $f,g \in \PM_p(\TM)$ with $g\succeq f$. Then by Strassen's theorem \cite{Strasse1965ExiProbMeas} there is a martingale kernel $\rho:\TM \to \PM(\R^d_v)$ with $g = f\rho$.
			
			Let $\pi: \TM \times (0,1) \to \PM(\R^d_v)$ be an admissible Markov kernel with $\Phi^\rel(g) \leq \int_0^1 \Phi(g\pi_s)\,ds + \eps$. Crucially, $\rho\pi_s$ is also an admissible Markov kernel.

			We check that
			\[
			\Phi^\rel(f) \leq \int_0^1 \Phi(f\rho\pi_s)\,ds = \int_0^1 \Phi(g\pi_s)\,ds \leq \Phi^\rel(g) + \eps.	
			\]
			
			Since $\eps>0$ was arbitrary, (3) follows.

			(4): First of all, we have the following inequalities
			\begin{align*}
				\Phi^\mathrm{conv}\leq \Phi \implies (\Phi^\mathrm{conv})^\inc\leq \Phi^\inc \implies (\Phi^\mathrm{conv})^\inc\leq (\Phi^\inc)^\mathrm{conv},
			\end{align*}
			where the last implication follows from the fact that $ (\Phi^\mathrm{conv})^\inc $ is convex. Indeed, let $ \Psi=\Phi^\mathrm{conv} $ and let $g\succeq f$ and $g'\succeq f'$. Then $(1-\lambda)g + \lambda g' \succeq (1-\lambda)f + \lambda f'$, since for a $v$-convex test function $\varphi:\TM\to \R$ we have
			\begin{align*}
				\langle (1-\lambda)g + \lambda g' , \varphi \rangle &= (1-\lambda)\langle g, \varphi \rangle + \lambda \langle g',\varphi\rangle 
				\\
				&\geq (1-\lambda) \langle f, \varphi \rangle + \lambda \langle f', \varphi \rangle = \langle (1-\lambda)f + \lambda f' , \varphi \rangle.
			\end{align*}
			Thus
			\begin{align*}
				\Psi^\inc((1-\lambda)f + \lambda f') \leq \Psi((1-\lambda)g + \lambda g') \leq (1-\lambda)\Psi(g) + \lambda \Psi(g').
			\end{align*}
			Minimizing over $g\succeq f$ and $g'\succeq f'$ yields convexity of $\Psi^\inc $.
			
			Furthermore, for a given convex $ \Psi $ we have $ (\Psi)^\inc=\Psi^\rel $. Indeed, for any admissible Markov kernel $\pi_s$ we have by Jensen's inequality
			\begin{align*}
				\int_0^1\Psi(f\pi_s)\,ds \geq \Psi\left(f\int_0^1 \pi_s\,ds\right).
			\end{align*}
			Thus we have that
			\begin{align*}
				\Psi^\rel(f) = \inf\{\Psi(f\pi)\,:\,\pi\in MK_p(\TM)\} = \inf\{\Psi(g)\,: \, g\succeq f\} = \Psi^\inc(f).	
			\end{align*}
			With this we infer by the monotony of the relaxation
			\begin{align*}
				\Phi^\mathrm{conv} \leq \Phi \implies (\Phi^\mathrm{conv})^\inc=(\Phi^\mathrm{conv})^\rel \leq \Phi^\rel \leq \Phi.
			\end{align*}
			This concludes the proof.
	\end{proof}

\section{Proof of the relaxation result}\label{sec: relaxation proof}

In this section we prove Theorem \ref{theorem: relaxation}. We first discuss the topology on $\PM_p(PS_T^p)$, showing that sequences of probabilities with bounded action are tight, before proving the relaxation theorem.

\subsection{Proof of compactness}

We show the following compactness result.
\begin{prop}\label{prop: compactness}
	Let $P^N\in \PM_p(PS_T^p)$ be a sequence with $\sup_{N\in\N} \langle P^N_0 , |x|^p \rangle< \infty$ and
	\begin{equation}
		\sup_{N\in\N} \int_0^T\langle P^N_{t,\dot t}, |v|^p \rangle\,dt < \infty.
	\end{equation}

	Then there is a subsequence (not relabeled) such that
	\begin{equation}
		\langle P^N, G \rangle \to \langle P,G \rangle
	\end{equation}
	for all $G:PS_T^p\to \R$ that are bounded and continuous with respect to the weak topology of $W^{1,p}([0,T];\R^d_x)$. We denote this convergence of probability measures by $P^N \weakstar P$. If $P^N \weakstar P$ then in particular $P^N_{s,t}\weakstar P_{s,t}$ in the narrow topology of $\PM(\R^d_x\times \R^d_x)$ for all $s,t\in[0,T]$.
\end{prop}
The statement is complicated slightly because $W^{1,p}$ with the weak topology is not a metrizable Polish space. However, all the balls in $W^{1,p}$ are metrizable, which turns out to be enough. Note that $P^N\weakstar P$ does not imply $P^N_{t,\dot t}\weakstar P_{t,\dot t}$ for almost all $t\in[0,T]$.

\begin{proof}
First we show tightness of the probability measures $P^N$. Consider the separable Banach space $L^p([0,T];\R^d_x)$, which contains $PS_T^p$ as a dense subspace. The sets
\[
B_R := \{\gamma\in PS_T^p\,:\,|\gamma_0|^p + \int_0^T |\dot \gamma|^p\,dt\leq R^p\}	
\]
are compact in $L^p$ by the Rellich Lemma. By Markov's inequality,
\[
R^p P^N(L^p\setminus B_R) \leq 	\int_0^T\langle P^N_{t,\dot t}, |v|^p \rangle\,dt + \langle P^N_0, |x|^p\rangle \leq C
\]
for all $N\in\N$ and all $R>0$, so that the $P^N$ are tight in $\PM(L^p)$. It follows from Prokhorov's theorem that there exists a limit probability measure $P\in \PM(L^p)$ such that for every $G:L^p\to \R$ that is bounded and continuous in the strong $L^p$ topology, we have
\[
\langle P^N,G \rangle \to \langle P,G \rangle.	
\]

In addition, $\lim_{R\to\infty} P(L^p\setminus B_R) = 0$, so that $P\in \PM(PS_T^p)$. If $G:PS_T^p\to \R$ is continuous under the weak topology of $W^{1,p}$ and bounded, we define $G_R:L^p\to \R$ As $G_R = G\circ P_R$ where $P_R:L^p\to B_R$ is the projection onto the convex compact set $B_R$. Since $P_R:L^p\to L^p$ is continuous in the strong topology and $\gamma_k \rightharpoonup \gamma$ in $W^{1,p}$ implies $\gamma_k \to \gamma$ in $L^p$, it follows that $G_R:L^p\to \R$ is strongly continuous and of course bounded.

We can estimate
\[
|\langle P^N -P, G\rangle| \leq |\langle P^N - P , G_R\rangle| + |\langle P^N - P , G- G_R\rangle|.
\]

The first term converges to zero since $G_R:L^p\to\R$ is strongly continuous. The second can be bounded independently of $N$ by
\[
	|\langle P^N - P , G- G_R\rangle| \leq \max|G| \left( P^N(L^p\setminus B_R) + P(L^p\setminus B_R) \right) \leq \max|G|/R^p.
\]

This shows that
\[
\langle P^N,G \rangle \to \langle P,G \rangle.	
\]

Now fix $s,t\in[0,T]$. We have to show that $\langle P^N_{s,t},g\rangle \to  \langle P_{s,t},g\rangle$ for all $g\in C_b(\R^d_x\times \R^d_x)$. This is clear, since the functional $G:PS_T^p\to \R$, $G(\gamma):= g(\gamma(s),\gamma(t))$ is weakly continuous.
\end{proof}

\subsection{Proof of the lower bound}
In this subsection we show that $ F^\rel(P) $ is an asymptotic lower bound for $ F(P^k) $ as $ P^k\rightharpoonup P $.

\begin{prop}\label{prop:LowerBound}
	Whenever $ P^k\weakstar P \in \PM_p(PS_T^p)$ we have
	\begin{align}\label{eq:PropLowerBound}
		\liminf_{k\to \infty} F(P^k) \geq F^\rel (P).
	\end{align}
\end{prop}

The proof uses a type of blow-up argument. Over short time intervals, curves in the support of the limit measure $P\in \PM_p(PS_T^p)$ are almost affine, thus their statistics $P_{t,\dot t}$ are almost determined by their couplings $P_{t,t+h}$. The need for relaxation arises precisely because as $P^k \weakstar P$, $P^k_{t,\dot t}$ does not converge narrowly to $P_{t,\dot t}$. However, $P^k_{t,t+h}$ does converge narrowly to $P_{t,t+h}$. By pushing the curves in the support of $P^k$ to the tangent space $TPS^p$ as in \eqref{eq: path relaxation}, we ensure that $P^k_{t,\dot t}$ are competitors to the relaxation for some statistics close to $P_{t,\dot t}$.

\begin{proof}
	We can assume that the limit $ \lim_{k\to \infty} F(P_k)<\infty $ exists and is finite.  We split the proof into several steps:

\emph{Step 1: Approximation by piecewise affine paths.} In this step, we replace all the curves in the support of $P$ by piecewise affine curves and ensure that the change in action is small.

Define for $N\in \N$ the equidistant partition of $[0,T]$ into $N$ intervals with endpoints $t_i^N:= \frac{Ti}{N}$, $i=0,\ldots,N$.
Define the projection $e^N:PS_T^p\to PS_T^{p,N}$ onto the piecewise affine curves through
\[
\begin{cases}
(e^N \gamma)(0) = \gamma(0) & \\
\frac{d}{dt}(e^N \gamma)(t) = \fint_{t_i^N}^{t_{i+1}^N}\dot \gamma(s)\,ds &\text{ for }t\in (t_i^N,t_{i+1}^N).
\end{cases}
\]
The space of such curves is denoted by $ PS_{T,N} $. Define $Q^N := (e_N)_\#P\in \PM_p(PS_T^p)$. Since $e_N\gamma \to \gamma$ strongly in $W^{1,p}([0,T];\R^d_x)$, we have by the dominated convergence theorem
\[
\lim_{N\to \infty}\int_0^T W_p^p(Q^N_{t,\dot t},P_{t,\dot t})\,dt \leq \lim_{N\to\infty} \E_P[\|e_N \gamma - \gamma\|_{W^{1,p}}^p] = 0,
\]
since $\E_P[\|\gamma\|_{W^{1,p}}^p]<\infty$. By the continuity of $\Phi^\rel$ in Proposition \ref{prop:PropertiesRelaxation}, it follows by H\"older's inequality that
\[
\begin{aligned}
&\lim_{N\to\infty} \int_0^T |\Phi^\rel(Q^N_{t,\dot t})-\Phi^\rel(P_{t,\dot t})|\,dt\\
 \leq & \lim_{N\to\infty} \int_0^T(\langle P_{t,\dot t} + Q^N_{t,\dot t},C+C |v|^p \rangle^{(p-1)/p} W_p(P_{t,\dot t},Q^N_{t,\dot t}))\,dt\\
 \leq & \lim_{N\to\infty} \left( \int_0^T \langle P_{t,\dot t},2C+2C|v|^p \rangle\,dt \right)^{(p-1)/p} \left(\int_0^T W_p^p (P_{t,\dot t},Q^N_{t,\dot t})\,dt\right)^{1/p}\\
 = & 0,
\end{aligned}
\]
where we used that $\int_0^T \langle Q^N_{t,\dot t},|v|^p \rangle\,dt \leq \int_0^T \langle P_{t,\dot t},|v|^p \rangle\,dt$. In short,
\begin{equation}\label{eq: piecewise}
	\lim_{N\to\infty} \int_0^T \Phi^\rel(Q^N_{t,\dot t})\,dt = \int_0^T \Phi^\rel(P_{t,\dot t})\,dt.
\end{equation}

\emph{Step 2: Approximation of the relaxed action by a Riemann sum.} We now show that we can approximate the action by a Riemann sum:
\begin{equation}\label{eq: Riemann sum}
\lim_{N\to\infty} \frac{T}{N} \sum_{i=0}^{N_1} \Phi^\rel(f^{N}_{i}) = \int_0^T \Phi^\rel(P_{t,\dot t})\,dt,
\end{equation}
where
\begin{equation}\label{eq:ProofLowerBoundInterpolation}
	f^{N}_{i} := \left[(x,y) \mapsto \left(x,\frac{N}{T}(y-x)\right)\right]_\# P_{t_i^N,t_{i+1}^N} \in \PM_p(\R^d_x\times \R^d_v).
\end{equation}

To show \eqref{eq: Riemann sum}, we work with the piecwise affine curves $Q^N$. In fact, since $Q_{t,\dot t}^N = [(x,v) \mapsto (x + (t-t_i^N)v,v)]_\# f_i^N$ for all $t\in (t_i^N,t_{i+1}^N)$, we have
\[
\sup_{t\in (t_i^N,t_{i+1}^N)} W_p^p(f_i^N, Q^N_{t,\dot t}) \leq \left(\frac{T}{N}\right)^{p-1} \langle f_i^N, |v|^p \rangle
\]
and consequently, using again the continuity of $\Phi^\rel$ from Proposition \ref{prop:PropertiesRelaxation},
\[
	\begin{aligned}
		&\left|\int_0^T \Phi^\rel(Q^N_{t,\dot t})\,dt - \frac{T}{N}\sum_{i=0}^{N-1}\Phi^\rel(f^N_i) \right|\\
		\leq & \frac{T}{N}\sum_{i=0}^{N-1} \langle f_i^N,2C + 2C|v|^p \rangle^{(p-1)/p} \left(\frac{T}{N}\right)^{(p-1)/p} \langle f_i^N, |v|^p \rangle^{1/p}\\
		\leq & \left(\frac{T}{N}\right)^{(p-1)/p} \int_0^T \langle P_{t,\dot t}, 2C + (2C+1)|v|^p \rangle\,dt.
	\end{aligned}
\]
Together with \eqref{eq: piecewise} we obtain \eqref{eq: Riemann sum}.

\emph{Step 3: $P^k$ as competitors.} We now introduce the sequence $P^k\weakstar P$. We note that for every fixed $N\in \N$ and $i=0,\ldots,N-1$, we have $P^k_{t_i^N,t_{i+1}^N} \weakstar P_{t_i^N,t_{i+1}^N}$, and
\[
f^{N,k}_i := \left[(x,y) \mapsto \left(x,\frac{N}{T}(y-x)\right)\right]_\# P^k_{t_i^N,t_{i+1}^N} \weakstar f^N_i.
\]

We shall show that
\begin{equation}\label{eq: pass to limit}
\frac{T}{N}	\Phi^\rel (f^{N,k}_i) \leq \int_{t_i^N}^{t_{i+1}^N} \Phi(P^k_{t,\dot t})\,dt + \left(\frac{T}{N}\right)^{(p-1)/p}\int_{t_i^N}^{t_{i+1}^N}\langle P^k_{t,\dot t}, C+C|v|^p \rangle\,dt.
\end{equation}
Once we manage to show \eqref{eq: pass to limit}, the proof of the lower bound is complete, since by summation
\[
	\begin{aligned}
\frac{T}{N}\sum_{i=0}^{N-1}\Phi^\rel(f^N_i)	 = & \lim_{k\to\infty} \frac{T}{N}\sum_{i=0}^{N-1}\Phi^\rel(f^{N,k}_i)\\
 \leq &  \liminf_{k\to \infty} \left\lbrace \int_0^T \Phi(P^k_{t,\dot t})\,dt  + \left(\frac{T}{N}\right)^{(p-1)/p} \int_0^T \langle P^k_{t,\dot t}, C+C|v|^p \rangle\,dt \right\rbrace .
	\end{aligned}
\]
Letting $N\to\infty$ together with \eqref{eq: Riemann sum} gives the desired lower bound.

Now to prove \eqref{eq: pass to limit}:
Fix $N\in \N$, $i=0,\ldots,N$.
We look for a competitor to the relaxation formula \eqref{eq: path relaxation} for the statistic $f^{N,k}_i\in \PM_p(\R^d_x\times \R^d_v)$. We find this competitor by turning the curves $\gamma : [t_i^N\to t_{i+1}^N]\to \R^d_x$ into curves $r\gamma\in TPS^p$
 using the map $r : PS^p_T \to TPS^p$,
\[
(r \gamma)(s) := \left(\gamma(t_i^N), \frac{\gamma(t_i^N + \frac{T}{N}s)}{\frac{T}{N}}\right).	
\]

Then $R^k:= r_\#P^k\in \PM_p(TPS^p)$ does indeed satisfy $R^k_1 = f^{N,k}_i$, so that by \eqref{eq: path relaxation} we have
\[
\Phi^\rel(f^{N,k}_i)\leq \int_0^1 \Phi(R^k_{s,\dot s})\,ds.	
\]

To show that $R^k_{s,\dot s}$ is close to $P^k_{t,\dot t}$ for $t(s):= t_i^N + \frac{Ts}{N}$, we estimate 
\[
\max_{s\in[0,1]} W_p^p(R^k_{s,\dot s}, P^k_{t,\dot t}|_{t=t(s)}) \leq \left(\frac{T}{N}\right)^{p-1} \fint_{t_i^N}^{t_{i+1}^N}\langle P^k_{t,\dot t}, |v|^p \rangle\,dt.
\]

With this, \eqref{eq: pass to limit} follows as in the previous estimates, this time using the continuity of $\Phi$ instead of $\Phi^\rel$.
\end{proof}

\subsection{Proof of the upper bound}
Finally, we prove the upper  bound in Theorem \ref{theorem: relaxation}, repeated here:
\begin{prop}\label{prop:UpperBound}
	Let $P\in \PM(PS_T^p)$. Then there is a sequence of measures $P^k\in \PM(PS_T^p)$ such that $P^k_{0,T} = P_{0,T}$, $P^k \weakstar P$, and
	\begin{align}\label{eq:PropUpperBound}
		\lim_{k\to \infty} F(P^k) = F^\rel (P).
	\end{align}
\end{prop}

\begin{proof}
For the proof, we modify the curves $\gamma\in \mathrm{supp} P \subset PS_T^p$, obtaining for every original curve a one-parameter family of curves $u^N_\theta\gamma\in PS_T^P$, $\theta \in \T^1 = \R/\Z$ with $u^N_\theta \gamma(t_i^N) = \gamma(t_i^N)$ for all $t_i^N = \frac{i}{N}T$.

We shall define the sequence $P^N \in \PM_p(PS_T^p)$ through
\begin{align}\label{eq:ProofUpperBoundDefPN}
	P^N := \E_P\left[\int_{\T^1}\delta_{u^N_\theta\gamma} \, d\theta\right].
\end{align}
We automatically obtain the boundary values $P^N_{0,T} = P_{0,T}$, and $P^N\weakstar P$ will hold as long as
\[
\sup_{N} \,  \E_{P}\int_{\T^1}\left[\int_0^T \left|\dfrac{d}{dt}u^N_\theta \gamma(t)\right|^p \, dt \,d\theta\right]	<\infty.
\]

The choice of $u^N_\theta \gamma$ must be such that
\[
\lim_{N\to\infty} 	\int_0^T \Phi(P^N_{t,\dot t})\,dt = \int_0^T \Phi^\rel(P_{t,\dot t})\,dt.
\]

\emph{Step 1: Approximation by piecewise affine curves.} This step is identical to Steps 1 and 2 in Proposition \ref{prop:LowerBound}. We obtain the measures $f_i^N\in \PM_p(\R^d_x\times \R^d_v)$ as in \eqref{eq:ProofLowerBoundInterpolation} such that
\begin{align}\label{eq:ProofUpperBoundLimitApproximation}
	\int_0^T \Phi^\rel(P_{t,\dot t})\,dt = \lim_{N\to\infty} \frac{T}{N} \sum_{i=0}^{N-1}\Phi^\rel(f_i^N).
\end{align}

\emph{Step 2: Choice of Markov kernels and modification of curves.} We choose by the definition of $\Phi^\rel$ \eqref{eq:Relaxation} for every $i=0,\ldots,N-1$ numbers $\lambda_{ij}^N\in[0,1]$ for $j=1,\ldots,M_i$ and Markov kernels $\pi_{ij}^N :\R^d_x\times \R^d_v \to \PM_p(\R^d_v)$ such that $\sum_{j=1}^{M_i} \lambda_{ij}^N = 1$, $\sum_{j=1}^{M_i} \lambda_{ij}^N\pi_{ij}^N$ is a martingale kernel, and
\begin{equation}\label{eq:ProofUpperBoundDiscreteApproximation}
	\frac{T}{N}\sum_{i=0}^{N-1} \sum_{j=1}^{M_i} \lambda_{ij}^N \Phi(f\pi_{ij}^N) \leq \frac{T}{N}\sum_{i=0}^{N-1} \Phi^\rel(f_i^N) + \frac{1}{N}.
\end{equation}

We represent each $\pi_{ij}^N$ through a random variable $X_{ij}^N:\R^d_x\times \R^d_v \times \T^1 \to \R^d_v$ such that
\begin{equation}
\int_{\T^1} \delta_{X_{ij}^N(x_0,v_0,\theta)}\,d\theta = \pi_{ij}^N(x_0,v_0)
\end{equation}
for all $x_0\in \R^d_x$, $v_0\in \R^d_v$, $i=0,\ldots,N$, $j=1,\ldots,M_i$. Note that we do not require independence.

This allows us to define the map $u^N:\T^1 \times PS_{T,N} \to PS_T^p$ defined through
\begin{equation}
	\begin{cases}
		u^N_\theta \gamma(0) = \gamma(0)& \\
		\frac{d}{dt}u^N_\theta \gamma(t) := X_{ij}^N\left(x_i^N,v_i^N,\frac{t-t_{ij}^N}{t_{i,j+1}^N-t_{i,j}^N} +\theta\right), &\text{ for }t\in(t_{ij}^N, t_{i,j+1}^N),
	\end{cases}
\end{equation}
where $x_i^N := \gamma(t_i^N)$, $v_i^N:= \frac{N}{T}(\gamma(t_{i+1}^N)-\gamma(t_i^N))$, $t_{i0}^N = t_i^N$, and $t_{i,j+1}^N = t_{i,j}^N + \lambda_{ij}^N\frac{T}{N}$. This choice ensures that for every $\gamma\in PS_T^p$ and every $\theta\in \T^1$, we have
\[
	u^N_\theta \gamma(t_{ij}^N) - u^N_\theta \gamma (t_{i,j-1}^N) = \lambda_{ij}^N \frac{T}{N} \int_{\T^1} X_{ij}(x_i^N,v_i^N,\theta)\,d\theta  =  \lambda_{ij}^N\frac{T}{N}\langle \pi_{ij}^N(x_i^N,v_i^N), v \rangle.
\]

Summing up over all $j=1,\ldots,M_i$ and using the martingale property yields
\[
u^N_\theta \gamma(t_{i+1}^N) - u^N_\theta\gamma(t_i^N) = \frac{T}{N} \sum_{j=1}^{M_i} \lambda_{ij}^M \langle \pi_{ij}^N(x_i^N,v_i^N), v \rangle = \frac{T}{N}v_i^N = \gamma(t_{i+1}^N) - \gamma(t_i^N).
\]
Let us mention that the estimate \eqref{eq: p moment} below shows that the defined paths are indeed in $ PS^p_T $. The above construction is schematically shown in Figure \ref{fig:pathsproof}. With the mappings $ u^N_\theta $ we can define measures $ P^N $ via \eqref{eq:ProofUpperBoundDefPN}.
\begin{figure}[!ht]
	\centering
	\includegraphics[width=0.9\linewidth]{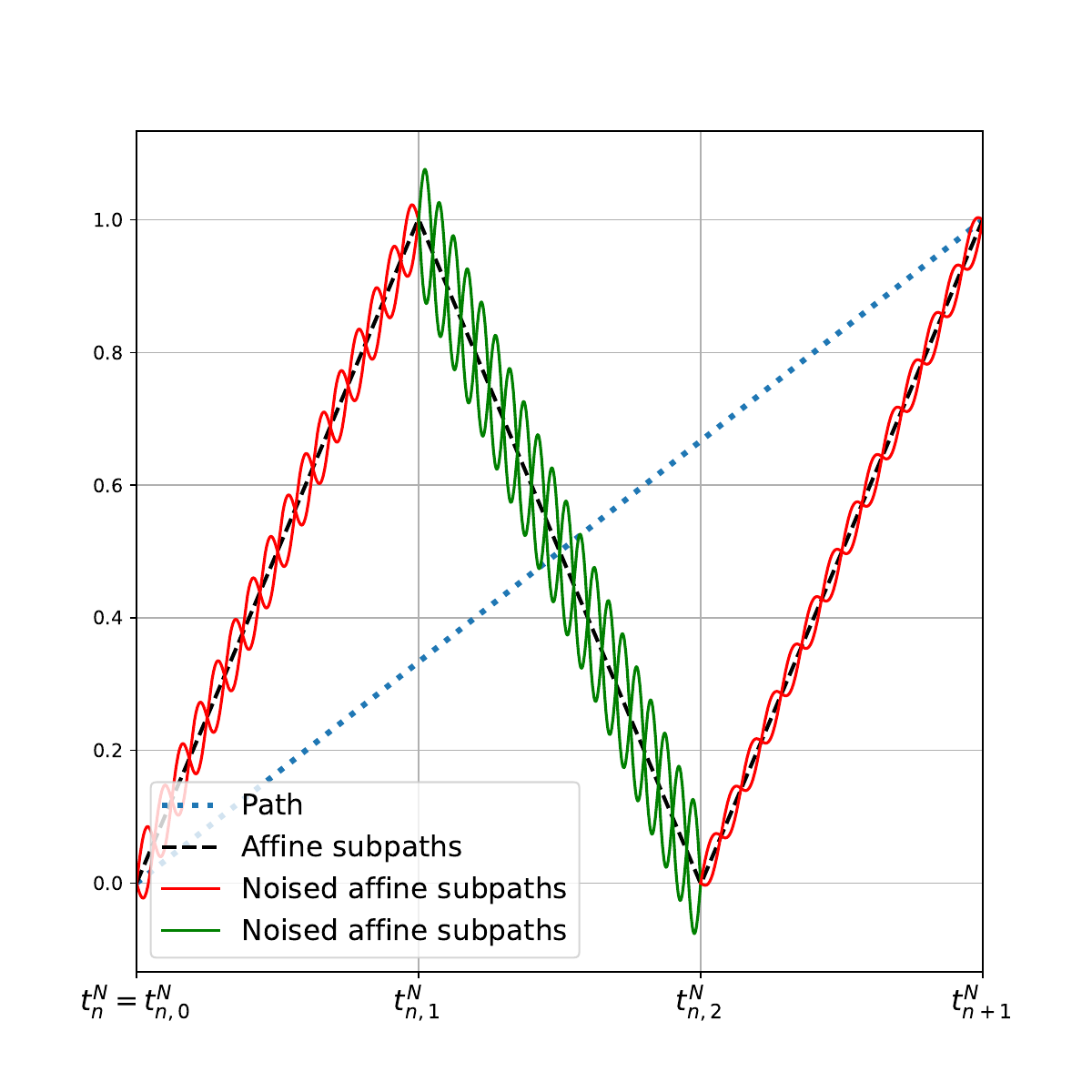}
	\caption{Schematic picture of paths constructed in Step 2 of Proof of Proposition \ref{prop:UpperBound}. One path drawn from the linear approximation $ P^N $ of the measure $ P $ is approximated by affine ``subpaths''. These ``subpaths'' are then over-layered by noise drawn from the martingale kernels on the subintervals $ [t^N_{n,i},t^N_{n,i+1}] $. Each ``subpath'' is given by the mean of the noise.}
	\label{fig:pathsproof}
\end{figure}

We finish this step by showing that $P^N \weakstar P$, which follows from the fact that
\begin{equation}\label{eq: p moment}
	\begin{aligned}
\E_P\left[\int_0^T \left|\dfrac{d}{dt}u^N_\theta\gamma(t)\right|^p \, dt\right ]  = & \frac{T}{N}\sum_{i=0}^{N-1} \sum_{j=1}^{M_i} \lambda_{ij}^N \langle f\pi_{ij}^N, |v|^p	\rangle \\
\leq CT+C\frac{T}{N}\sum_{i=1}^N\Phi^\rel(f_i^N) \leq & \int_0^T \langle P_{t,\dot t},C+C|v|^p \rangle\,dt 
	\end{aligned}
\end{equation}
by the growth condition (A1) of $\Phi$ and the same growth condition of $\Phi^\rel$, which holds by Proposition \ref{prop:PropertiesRelaxation}.

\emph{Step 3: Upper bound of the action.} We have to show that \begin{align}\label{eq:ProofUpperBoundAction}
	\limsup_{N\to\infty} \int_0^T \Phi(P^N_{t,\dot t})\,dt \leq \int_0^T \Phi^\rel(P_{t,\dot t})\,dt
\end{align}
With Proposition \ref{prop:LowerBound} this concludes the proof. In order to do so, we estimate
\begin{equation}\label{eq: modification wasserstein}
	W_p^p(P^N_{t,\dot t},f_i^N\pi_{ij}^N)\leq \left(\frac{T}{N}\right)^{p-1} \int_0^T \langle P_{t,\dot t}, C+C|v|^p \rangle\,dt 
\end{equation}
for all $t\in (t_{i,j-1}^N,t_{ij}^N)$. If \eqref{eq: modification wasserstein} holds, we can then use the continuity (A2) and growth condition (A1) of $\Phi$ to estimate
\begin{equation}
\sum_{i=0}^{N-1} \sum_{j=1}^{M_i}\int_{t_{i,j-1}^N}^{t_{ij}^N} |\Phi(P^N_{t,\dot t}) -\Phi(f_i^N\pi_{ij}^N)|\,dt \leq \left(\frac{T}{N}\right)^{(p-1)/p}\int_0^T \langle P_{t,\dot t}, C+C|v|^p \rangle\,dt.
\end{equation}
Due to \eqref{eq:ProofUpperBoundLimitApproximation} and \eqref{eq:ProofUpperBoundDiscreteApproximation} this proves the statement \eqref{eq:ProofUpperBoundAction}.

Now to prove \eqref{eq: modification wasserstein}: We can write
\begin{align*}
	f_i^N\pi_{ij}^N &= \int_{\T^1} \E_P[\delta_{(x_i^N, X_{ij}^N(x_i^N,v_i^N,\theta))}]\,d\theta
	\\
	P^N_{t,\dot t} &= \int_{\T^1}\E_P [\delta_{(u^N_\theta\gamma (t), X_{ij}^N(x_i^N,v_i^N,\theta))}]\,d\theta
\end{align*}
Thus we may estimate
\begin{equation}
	\begin{aligned}
W_p^p(f_i^N\pi_{ij}^N, P^N_{t,\dot t} ) \leq &\int_{\T^1} \E_P[|x_i^N - u^N_\theta\gamma (t)|^p]\,d\theta \\
\leq& \left(\frac{T}{N}\right)^{p-1} \int_{\T^1}\E_p\left[\int_0^T |\dot \gamma|^p + \left|\dfrac{d}{dt}u^N_\theta \gamma(t) \right|^p \,dt\right]\,d\theta\\
\leq & \left(\frac{T}{N}\right)^{p-1}\int_0^T \langle P_{t,\dot t}, C+C|v|^p\rangle \,dt,
	\end{aligned}
\end{equation}
where we used \eqref{eq: p moment}. This completes the proof of the upper bound.
\end{proof}

\section{The limit of the $N$-particle problem}\label{sec:NParticleProblem}
In this section, we prove Theorem \ref{theorem: N body}.
\begin{proof}[Proof of Theorem \ref{theorem: N body}]
	A minimizer $P^N = \frac1N \sum_{i=1}^N \delta_{x^i}\in \PM(PS_T^p)$ to the $N$-body problem exists by the direct method of the calculus of variations.
	By Theorem \ref{theorem: relaxation} and Proposition \ref{prop: compactness}, any sequence of minimizers has a subsequence converging to some $P\in \PM(PS_T^p)$ with $F^\rel(P) \leq \liminf_{N\to\infty} F(P^N)$.

	We have to show that there is $P^N =  \frac1N \sum_{i=1}^N \delta_{x^i}\in \PM(PS_T^p)$ with $P^N_{0,T} = \pi^N$ and  $\limsup_{N\to\infty} F(P^N) \leq F^\rel(P)$. To see this, first approximate $P$ by a sequence $Q^k \weakstar P$ such that $\lim_{k\to\infty} F(Q^k) = F^\rel(P)$. Choose a minimizer $Q^{k,N} = \frac1N \sum_{i=1}^N \delta_{x^i}\in \PM(PS_T^p)$ to the semidiscrete transport problem on the path space, i.e.
	\[
	Q^{k,N}\in\arg\min_Q	\left\lbrace W_{p,\|\cdot\|_{W^{1,p}}}^p(Q,Q^k)\,:\,Q = \dfrac{1}{N}\sum_{i=1}^N \delta_{x_i}\in \PM(PS^P_p) \right\rbrace ,
	\]
	where
	\[
		\begin{aligned}
	W_{p,\|\cdot\|_{W^{1,p}}}^p(Q,Q^k) := &\inf \left\lbrace \int_{PS^T_p \times PS^T_p} \int_0^T \left( |x(t)-y(t)|^p + |\dot x(t) - \dot y(t)|^p \right)  \, dt\gamma(dx,dy) \,: \right. 
	\\
	&\qquad \left. \gamma\in \PM(PS^T_p \times PS^T_p)\text{ is a coupling between }Q,Q^k \right\rbrace 	\end{aligned}
	\]
	is the $p$-Wasserstein distance in the strong $W^{1,p}$-norm of path space. By optimal transport theory \cite{ambrosio2013user}, we have
	\[
	\lim_{N\to\infty} \int_0^1 W_p^p(Q^{k,N}_{t,\dot t}, Q^{k}_{t,\dot t})\,dt \leq \lim_{N\to\infty} 	W_{p,\|\cdot\|_{W^{1,p}}}^p(Q^{k,N},Q^k) = 0,
	\]
	so that $Q^{k,N} \weakstar Q^k$ and by (A2) $\lim_{N\to \infty}F(Q^{k,N}) = F(Q^k)$	for every $k\in\N$. Taking a diagonal sequence $k(N)\to\infty$, we can assure that $Q^{k(N),N}\weakstar P$ and $\lim_{N\to\infty} F(Q^{k(N),N}) = F^\rel(P)$. In order to also realize the boundary values $P^N_{0,T} = \pi^N$, we realize that
	\begin{align*}
		\lim_{N\to\infty}W_p^p(Q^{k(N),N}_{0,T},\pi^N) = 0.
	\end{align*}
	Solving the assignment problem between these two empirical measures yields a transport map $S^N:\mathrm{supp}(Q^{k(N),N}_{0,T}) \to \mathrm{supp}(\pi^N)$, which can be factored into two maps $S_0^N,S_T^N:\R^d_x \to \R^d_x$. Then take $P^N := \frac1N \sum_{i=1}^N \tilde x_i$, where
	\[
	\tilde x_i(t) := x_i(t) + (S_0^N(x_i(0))-S_0^N(x_i(0))) + \frac{t}{T}((S_T^N(x_i(T) - S_0^N(x_i(0)))-(x_i(T) - x_i(0)))).	
	\]

	We check that $P^N_{0,T} = S^N_\# Q^{k(N),N}_{0,T} = \pi^N$, $P^N \weakstar P$, and $\lim_{N\to\infty} F(P^N) = F^{\rel}(P)$.
\end{proof}

\section{The interacting particle optimal transport problem}\label{sec:InteractingOTP}

Lastly we consider the interacting particle optimal transport problem, where instead of prescribing a coupling $\Gamma_b\in \PM_p(\R^d_x\times \R^d_x)$, we only prescribe the initial and final distributions $\mu_0,\mu_T\in \PM_p(\R^d_x)$ and minimize the action
\[
E_T(\mu_0,\mu_T) := \inf\left\lbrace \int_0^T \Phi(P_{t,\dot t})\,dt\,:\,P_0 = \mu_0,P_T = \mu_T\right\rbrace .	
\]
By Theorem \ref{theorem: relaxation}, we can rewrite $E_T$ as
\[
E_T(\mu_0,\mu_T) = \min\left\lbrace 	\int_0^T \Phi^\rel(P_{t,\dot t})\,dt\,:\,P_0 = \mu_0,P_T = \mu_T\right\rbrace .	
\]
We claim that $E_T$ has the following Eulerian version.
\begin{theorem}\label{theorem: optimal transport}
Assume that $\Phi$ satisfies (A1) and (A2). Define for a mass density $\rho\in \PM(\R^d_x)$ and a Borel-measurable velocity field $V:\R^d_x\to \R^d_v$ the particle statistics $f(\rho,V):= E_\rho[\delta_{(x,V(x))}]\in \PM(\TM)$ and the energy $\Psi(\rho,V):= \Phi^\rel(f(\rho,V))$. Then  $\Psi$ satisfies the growth and continuity conditions
\begin{itemize}
	\item [\textbf{(D1)}] $\displaystyle{\dualbra{\rho}{-C+c |V|^p} \leq \Psi(\rho,V) \leq \dualbra{\rho}{C+C |V|^p}}$
	\\ 
	for all $\rho\in \PM(\R^d_x)$, $V:\R^d_x\to\R^d_v$ Borel-measurable.
	\vspace*{0.2cm}
	\item [\textbf{(D2)}] $|\Psi(\rho,V) - \Psi(\rho,V')| \leq \dualbra{\rho}{C+C|V|^p+C|V'|^p}^{(p-1)/p}  \langle \rho, |V-V'|^p \rangle ^{1/p}$ 
	\\ 
	for all $\rho\in \PM(\R^d_x)$ and all $V,V':\R^d_x \to \R^d_v$ Borel-measurable.
\end{itemize}

In addition, for all $\mu_0,\mu_1\in \PM_p(\R^d_x)$ and all $T>0$ we have the Eulerian representation
\begin{equation}\label{eq: Eulerian}
	E_T(\mu_0,\mu_T) = \min\left\{ \int_0^T \Psi(\rho_t,V_t)\,dt\,:\,(\rho,V)\in \mathcal{CE}^p_T(\R^d_x), \rho_0 = \mu_0,\rho_T = \mu_T\right\},
\end{equation}
where $\mathcal{CE}^p_T(\R^d_x)$ denotes all Eulerian solutions to the continuity equation
\begin{equation}
	\begin{aligned}
\mathcal{CE}^p_T(\R^d) := \{(\rho,V)\,&:\,\rho\in \mathcal{M}_+((0,T)\times \R^d_x), V:(0,T)\times \R^d_x \to \R^d_v \text{ Borel-measurable}, \\
&\langle\rho,|V|^p\rangle < \infty, \partial_t \rho + \dive (V\rho) = 0\text{ in }\mathcal{D'}((0,T)\times \R^d_x)\}.
	\end{aligned}
\end{equation}
\end{theorem}

Before we prove the Eulerian representation \eqref{eq: Eulerian}, let us observe that $\int_0^T \Psi(\rho_t,V_t)\,dt$ admits only particle statistics  where the velocity at almost every $x\in \R^d_x$ is constant, while the general formula $\int_0^T \Phi^\rel(P_{t,\dot t})\,dt$ clearly allows for particles of different velocities at every point.

However, since $\Phi^\rel$ is increasing in the convex order, replacing any particle statistics $f\in \PM(\TM)$ with its Eulerian collapse $\rho = f_x, V(x) = E_f[v|x]$ yields $\Psi(\rho,V)\leq \Phi^\rel(f)$.

Also note that all pairs $(\rho,V)\in \mathcal{CE}^p_T$ form H\"older-continuous curves $t\mapsto \rho_t$ in the Wasserstein space $(\PM_p(\R^d_x),W_p)$, with 
\begin{align*}
	W_p(\rho_{t_0},\rho_{t_1}) \leq (t_1-t_0)^{(p-1)/p}\left(\int_{t_0}^{t_1} \int_{\R^d_x}|V_t(x)|^p \rho_t(dx)\,dt\right)^{1/p}.
\end{align*}
A proof of the H\"older-continuity is found in e.g. \cite[Proposition 2.30]{ambrosio2013user}. It allows us to make sense of the boundary values $\rho_0 = \mu_0, \rho_T = \mu_T$.
\begin{proof}
The growth and continuity conditions (D1), (D2) follow immediately from (A1), (A2). We first show the inequality 
\[
E_T(\mu_0,\mu_T) \geq \min\left\{ \int_0^T \Psi(\rho_t,V_t)\,dt\,:\,(\rho,V)\in \mathcal{CE}^p_T(\R^d_x), \rho_0 = \mu_0,\rho_T = \mu_T\right\}
\]

To do so, take any candidate $P\in \PM(PS^T_p)$ with $P_0 =\mu_0,P_T = \mu_T$ and define $\rho_t = P_t$, $V_t(x) = E_P[\dot x(t)|x(t)=x]$ for $\rho_t$-almost every $x\in \R^d_x$, and extend $V_t$ by zero outside of $\supp(\rho_t)$. As discussed earlier, we have $P_{t,\dot t} \succeq  f(\rho_t,V_t) $ in the convex order, so that $\Psi(\rho_t,V_t) \leq \Phi^\rel(P_{t,\dot t})$ for almost every $t\in(0,T)$. We show that $(\rho,V)\in \mathcal{CE}^p_T$. First,
\[
\langle \rho,-C+c |V|^p	\rangle \leq \int_0^T \Psi(\rho_t,V_t)\,dt <\infty.
\]

Second, take any test function $\varphi\in C_c^\infty((0,T)\times \R^d_x)$. Then 
\[
\begin{aligned}
	\int_0^T \int_{\R^d_x} \partial_t \varphi(t,x)  \rho_t(dx)\,dt = & \E_P\left[\int_0^T \partial_t \varphi(t,x(t))\,dt\right]
	\\
	= & \E_P\left[\int_0^T \frac{d}{dt}[\varphi(t,x(t))] - \nabla_x \varphi(t,x(t))\cdot \dot x(t)\,dt\right]
	\\
	= & \E_P\left[\int_0^T -\nabla_x\varphi(t,x(t))\cdot E_P[\dot x(t)|x(t)]\right]
	\\
	= & \int_0^T\int_{\R^d_x} - \nabla_x \varphi(t,x) \cdot V(x) \rho_t(dx)\, dt.
\end{aligned}
\]
Since $\varphi$ was arbitrary, this shows that $(\rho,V)\in \mathcal{CE}^p_T(\R^d_x)$, and the inequality
\begin{align*}
	E_T(\mu_0,\mu_T) \geq \min\left\lbrace \int_0^T(\rho_t,V_t)\,dt\right\rbrace .
\end{align*}

To show the other inequality,
\[
E_T(\mu_0,\mu_T) \leq \min\left\{ \int_0^T \Psi(\rho_t,V_t)\,dt\,:\,(\rho,V)\in \mathcal{CE}^p_T(\R^d_x), \rho_0 = \mu_0,\rho_T = \mu_T\right\},
\]
start with a solution $(\rho,V)\in \mathcal{CE}^p_T(\R^d_x)$, with $\rho_0 = \mu_0$ and $\rho_T = \mu_T$. By Smirnov's superposition principle, see e.g. \cite{ambrosio2013user}, there is a probability measure $P\in  \PM(PS_T^p)$ with $P_{t,\dot t} = f(\rho_t,V_t)$ for almost every $t\in(0,T)$, and $P_t = \rho_t$ for every $t\in[0,T]$ by H\"older-continuity. In particular, $P_0 = \mu_0, P_T = \mu_T$, and
\[
\int_0^T \Phi^\rel(P_{t,\dot t})\,dt = \int_0^T	\Psi(\rho_t,V_t)\,dt,
\] 
completing the proof.
\end{proof}

We conclude the article by specifying the Euler-Lagrange equation for the interacting particle optimal transport problem. It is known from classical optimal transport theory that the velocity field at all times can be determined through the viscosity solution of a Hamilton-Jacobi-Bellman equation, see e.g. \cite[Box 8.3]{santambrogio2015optimal}.

For interacting particles, we make use of Proposition \ref{prop:ELEquation}, which says that any measure optimizing the interaction action also optimizes the marginal noninteracting action, for which we know that the dynamic programming principle holds.

\begin{prop}\label{prop: HJB equation}
	Let $T>0$, $\mu_0,\mu_T\in \PM_2(\R^d_x)$. Let $\Phi(f) = \langle f,\psi \rangle + \frac12\langle f, U\ast f\rangle$ satisfy (C1) and (C2). Let $(\rho,V)\in \mathcal{CE}_T^2$ be a minimizer of $\int_0^T \Psi(\rho_t,V_t)\,dt$ subject to $\rho_0 = \mu_0,\rho_T = \mu_T$. Assume that $\rho,V\in C^2([0,T]\times \R^d_x)$ and that $\rho>0$ everywhere. Define the Lagrangian $L_t:= L[f(\rho_t,V_t)]\in C^2([0,T]\times \TM)$ as in Theorem \ref{thm:CauchyProblem} and its Legendre transform $L_t^*(x,p):=\sup_{v\in\R^d_v} p\cdot v - L_t(x,v)$. Then there exists a potential $\Xi\in C^2([0,T]\times \R^d_x)$ solving the Hamilton-Jacobi-Bellman equation
	\begin{equation}\label{eq: HJB}
		\begin{cases}
			V(t,x) = \nabla_p L_t^*(x,\nabla_x\Xi(t,x))\\
			\partial_t \Xi(t,x) + L_t^*(x,\nabla_x \Xi(t,x)) = 0.
		\end{cases}
	\end{equation}
\end{prop}

While this result, known as the dynamic programming principle, is well-known for classical optimal transport, we will give a proof in our case. Note that $\Xi$ is called the value function in optimal control theory, and in classical optimal tranport there is the Kantorovich dual formulation, see e.g. \cite[Remark 6.2]{santambrogio2015optimal}, here this does not apply, since $L$ is only the marginal cost and there is no equality between the marginal cost $\int_0^T \int_{\R^d_x}L_t(x,V(x))\,\rho_t(dx)dt$ and the actual interaction action $\int_0^T \Psi(\rho_t,V_t)\,dt$, despite both sharing minimizers.

\begin{proof}
We first show that for every $t\in[0,T]$ there is a potential $\tilde \Xi(t,\cdot)\in C^1(\R^d_x)$ such that $V(t,x) = \nabla_p L_t^*(x,\nabla_x \tilde \Xi(t,x))$. By Legendre duality, this is equivalent to
\begin{equation}\label{eq: curlfree}
	\nabla_v L_t(x,V(t,x)) = \nabla_x \tilde\Xi(t,x).
\end{equation}

To see \eqref{eq: curlfree}, take a divergence-free test vector field $X\in C_c^\infty(\R^d_x;\R^d_v)$. Then $(\rho,V+\eps\frac{X}{\rho})\in \mathcal{CE}_T^2$ for all $\eps\in \R$, since $\rho>0$. By optimality
\begin{equation}
	 \begin{aligned}
		0 = & \frac{d}{d\eps}\bigl\vert_{\eps=0} \Psi\left(\rho,V+\eps\frac{X}{\rho}\right)
		\\
		= 	& \int_{\R^d_x}\nabla_v L_t(x,V(t,x)) \cdot X(x)\,dx.
	 \end{aligned}
\end{equation}

Since $\nabla_v L_t(x,V(t,x))$ is orthogonal to the divergence-free vector fields, it must be the gradient of some potential $\tilde \Xi(t,\cdot)\in C^2$. By subtracting a constant, we can assume that $\tilde \Xi(t,0) = 0$. We now check that $\tilde \Xi$ solves
\begin{equation}\label{eq: HJB gradient}
	\partial_t \nabla_x \tilde \Xi(t,x) = -\nabla_x (L_t^*(x,\nabla_x \tilde \Xi(t,x))).
\end{equation}

From \eqref{eq: HJB gradient} it follows that $\tilde \Xi$ itself solves $\partial_t \tilde\Xi(t,x) + L_t^*(x,\nabla_x \tilde \Xi(t,x)) = c(t)$ for some $c\in C^0([0,T])$. Defining $\Xi(t,x) := \tilde\Xi(t,x) - \int_0^t c(s)\,ds$ gives the solution to \eqref{eq: HJB}.

To show \eqref{eq: HJB gradient} itself, we define $p(t,x):= \nabla_v L_t(x,V(t,x)) = \nabla_x \tilde \Xi_t(x)$. Using Proposition \ref{prop:ELEquation} as well as similar arguments as in Theorem \ref{theorem: optimal transport} one can show that $p$ solves the inhomogeneous transport equation
\begin{align*}
	\partial_t(\rho \, p) + \dive( \rho \, p\otimes V) = \rho(t,x)\, \nabla_xL_t(x,V(t,x)).
\end{align*}
Here, we used the notation $ (p\otimes V)_{ij} = p_uiV_j $ and for a matrix-valued function $ M:\R^d\to \R^{d\times d} $, $ \dive(M)\in \R^d $ with $ \dive(M)_i = \sum_j \partial_{x_j}M_{ij} $. Due to the transport equation $ \partial_t\rho +\dive(\rho V)=0 $ and the assumption $ \rho>0$ we obtain
\begin{align}\label{eq: Equation for p}
	\partial_tp +(V\cdot \nabla_x) p = \nabla_xL_t(x,V(t,x))
\end{align}
with $ [(V\cdot \nabla_x) p]_i  = \sum_j V_j\partial_{x_j} p_i $.

Now, the left-hand side of \eqref{eq: HJB gradient} writes
\begin{equation}\label{eq: HJB gradient left}
	\partial_t \nabla_x \tilde \Xi(t,x) = \partial_t p(t,x) = \nabla_x L_t(x,V(t,x)) - (V\cdot \nabla_x) p(t,x).
\end{equation}

Meanwhile, the right-hand side of \eqref{eq: HJB gradient} is
\begin{align}\label{eq: HJB gradient right}
\begin{split}
	\nabla_x(L_t^*(x,\nabla_x \tilde \Xi(t,x))) = &\nabla_x(V(t,x)\cdot p(t,x) - L_t(x,V(t,x)))
	 \\
	= & \nabla_x V(t,x) \, p(t,x) + \nabla_xp(t,x) \, V(t,x) 
	\\
	&\quad - \nabla_x L_t(x,V(t,x)) - \nabla_x V(t,x) \, \nabla_v L_t(x,V(t,x)) 
	\\
	= &  \nabla_xp(t,x) \, V(t,x) - \nabla_x L_t(x,V(t,x)),
\end{split}	
\end{align}
since the first and last term cancel each other. Here, we used the notation for the gradient of the vector field $ (\nabla_x V)_{ij} = \partial_{x_j}V_i $. Observe that by definition of $ p $ we have 
\begin{align*}
	\nabla_xp(t,x) \, V(t,x) = D^2_x\tilde{\Xi}(t,x) \, V(t,x)  = (V\cdot \nabla_x) p(t,x).
\end{align*}
Thus, combining \eqref{eq: HJB gradient left} and \eqref{eq: HJB gradient right} yields \eqref{eq: HJB gradient}.
\end{proof}

Further investigation into the nonsmooth situation might lead to additional insight into viscosity solutions of the Hamilton-Jacobi-Bellman equation \eqref{eq: HJB}. We finally note that taking the time derivative of the first equation in \eqref{eq: HJB} yields the compressible Euler equation for $V$. More precisely, we take the time-derivative of $ \nabla_x\Xi(t,x) = \nabla_vL_t(x,V(t,x)) $ yielding
\begin{align*}
	D^2_vL_t(x,V(t,x)) \partial_t V(t,x) + \nabla_v\partial_tL_t(x,V(t,x)) = \partial_t \nabla_x\Xi(t,x).
\end{align*}
For the right hand side we use \eqref{eq: Equation for p} with $ p = \nabla_x\tilde{\Xi} = \nabla_x \Xi=\nabla_vL_t(x,V(t,x))  $ to get
\begin{align*}
	\partial_t\nabla_x \Xi = -(V\cdot \nabla_x) [\nabla_vL_t(x,V(t,x))] + \nabla_xL_t(x,V(t,x)).
\end{align*}
We hence obtain
\begin{align*}
	\partial_t V(t,x) +  (V \cdot \nabla_x ) V(t,x) &= A[f(\rho_t,V_t)](x,V(t,x))
	\\
	A[f(\rho_t,V_t)](x,V(t,x)) &= D^2_vL_t(x,V(t,x))^{-1}\big[  \nabla_xL_t(x,V(t,x)) 
	\\
	&\quad - \nabla_x\nabla_vL_t(x,V(t,x))\, V(t,x) - \nabla_v\partial_tL_t(x,V(t,x)) \big].
\end{align*}
Here, $A[f]:\TM \to \R^d_a$ is the acceleration as it appears in Theorem \ref{thm:CauchyProblem}. However, not all solutions to the Euler equation stem from minimizers of the interacting particle optimal transport problem. Namely, only those solutions $(\rho,V)\in\mathcal{CE}_T^2$ where $p(t,x) = \nabla_v L_t(x,V(t,x))$ is curl-free, i.e. a gradient.

\subsection*{Acknowledgment} The authors gratefully acknowledge the support by the Deutsche Forschungsgemeinschaft (DFG) through the Collaborative Research Center "The mathematics of emerging effects" (CRC 1060, Project-ID 211504053). B. Kepka is funded by the Bonn International Graduate School of Mathematics at the Hausdorff Center for Mathematics (EXC 2047/1, Project-ID 390685813).

%
%

\bibliographystyle{plain}
\bibliography{References}

\end{document}